\documentclass{article}

\usepackage{amsmath, amsfonts, amsthm, amssymb}
\usepackage{enumerate}
\usepackage{tikz}
\usetikzlibrary{matrix,arrows}
\usepackage{hyperref}

\DeclareMathOperator{\Aut}{Aut}
\DeclareMathOperator{\GL}{GL}
\DeclareMathOperator{\Ind}{Ind}

\DeclareMathOperator{\davenport}{\mathsf{D}}
\DeclareMathOperator{\field}{\mathbb{F}}
\DeclareMathOperator{\supp}{supp}

\theoremstyle{plain}
\newtheorem{proposition}{Proposition}[section]
\newtheorem{theorem}[proposition]{Theorem} 
\newtheorem{lemma}[proposition]{Lemma}

\newtheorem{corollary}[proposition]{Corollary}

\theoremstyle{definition}
\newtheorem{definition}[proposition]{Definition}
\newtheorem{remark}[proposition]{Remark}

\title{Groups with large Noether bound}

\author{K\'alm\'an Cziszter $^a$ 
\thanks{The paper is based on results from the PhD thesis of the first author written at the Central European University.}
\qquad and \qquad M\'aty\'as Domokos $^b$ 
\thanks{The second author is partially supported by OTKA  NK81203 and K101515.}
}
\date{}

\begin{document}

\maketitle

{\small \begin{center} 
$^a$ Central European University, Department of Mathematics and its Applications, 
N\'ador u. 9, 1051 Budapest, Hungary \\
Email: cziszter\_kalman-sandor@ceu-budapest.edu \\
 $^b$ R\'enyi Institute of Mathematics, Hungarian Academy of Sciences,\\
 Re\'altanoda u. 13-15, 1053 Budapest, Hungary \\
Email: domokos.matyas@renyi.mta.hu
\end{center}
}

\begin{abstract}
The finite groups having an indecomposable polynomial invariant of degree  at least  half the order of the group are classified. 
It turns out that |apart from four sporadic exceptions| these are exactly the groups with a cyclic subgroup of index at most two.
\end{abstract}

\maketitle


\section{Introduction}  
\subsection{Outline of the main results} 

Let $G$ be a finite group and $V$  a $G$-module of finite dimension over a  field $\field$. 
By a classical theorem of E. Noether \cite{noether:1926}  the \emph{algebra of polynomial invariants} on $V$,  
denoted by $\field[V]^G$, is finitely generated. Set 
\begin{align*}
\beta(G,V)&:=\min\{d\in \mathbb{N}\mid \field[V]^G\text{ is generated by elements of degree at most } d\},  \\
\beta(G)&:=\sup\{\beta(G,V)\mid V\mbox{ is a finite dimensional }G\text{-module over } \field \}.
\end{align*}
The famous theorem on the {\it Noether bound} asserts that  
\begin{align}\label{noetherbound} \beta(G)\leq |G| \end{align}
provided that $\mathrm{char}(\field)$ does not divide the order of $G$
(see Noether \cite{noether:1916} in characteristic $0$ and Fleischmann \cite{fleischmann},  Fogarty \cite{fogarty} in positive characteristic). 
Schmid  proved in \cite{schmid} that  over the field of complex numbers $\beta(G)=|G|$ holds only when $G$ is cyclic. 
This was sharpened by Domokos and Heged\H us  in \cite{domokos-hegedus}  by proving  
that $\beta(G)\leq \frac 34|G|$ for all non-cyclic $G$;  
the result was extended to non-modular positive characteristic by Sezer \cite{sezer}. 
The constant $3/4$ is optimal here. 
On the other hand, a straightforward lower bound on $\beta(G)$ can be obtained 
based  on the result of Schmid in \cite{schmid}, that
$\beta(G)\geq\beta(H)$ holds for any subgroups $H$ of $G$. 
In particular, 
$\beta(G)$ is bounded from below by the maximal order of the elements in $G$. Therefore $\beta(G)\geq \frac 12|G|$
whenever $G$ contains a cyclic subgroup of index two,
and obviously there are infinitely many isomorphism classes of such non-cyclic groups. 
The  main result of the present article is that 
|apart from four sporadic exceptions| these are the only groups for which the ratio of the Noether number and the group order is so large:
 
\begin{theorem}\label{thm:main} 
For a finite group $G$ with order not divisible by $\mathrm{char} (\field)$   we have $\beta(G)\geq \frac 12|G|$ 
if and only if $G$ has a cyclic subgroup of index at most two, or 
$G$ is isomorphic to  
$Z_3\times Z_3$, 
$Z_2\times Z_2\times Z_2$, 
the alternating group $A_4$, or 
the binary tetrahedral group $\tilde{A_4}$. 
\end{theorem}

This theorem  is a novelty  even for the case $\field=\mathbb{C}$. 
The main technical tool of its proof is a generalization of the Noether number
which allows us to formulate some reduction lemmata in Section~\ref{sec:red}
that can be used to infer  estimates on the Noether number of a group
from the knowledge of the (generalized) Noether number of its subgroups and homomorphic images. 
Theorem~\ref{thm:structure} then isolates a list of some groups such that an arbitrary finite group $G$ must contain one of them 
as a subgroup or a subquotient, 
unless $G$ contains a cyclic subgroup of index at most two. Finally, the proof is made complete in Sections 2--3, 
where we compute or estimate the (generalized) Noether number for the particular groups on this list.

The quest for degree bounds has always been in the focus of invariant theory. A practical motivation  is that good initial degree bounds may potentially decrease the running time of algorithms to compute generators of invariant rings. On the other hand, the exact value of the Noether bound is known only for very few groups. 
To indicate the difficulties we mention the paper of Dixmier \cite{dixmier}, investigating the Noether number for irreducible representations of the symmetric group of degree $5$. It can be seen in the present paper as well that the discussion of some small groups, the estimation of the Noether bound takes relatively large space 
(especially where the exact value is found). 

We finish the introduction by noting that the constant $1/2$ in Theorem~\ref{thm:main} has a remarkable theoretical  status. 
In a parallel paper \cite{CzD:2} we determine the (generalized) Noether number for each non-cyclic group $G$ with a cyclic subgroup of index $2$: it turns out that 
for such a $G$ we have $\beta(G)-\frac 12|G|\in\{1,2\}$. Consequently, for any $c>1/2$, up to isomorphism  there are only finitely many non-cyclic groups 
$G$ with $\beta(G)/|G|>c$, whereas there are infinitely many isomorphism classes of non-cyclic groups $G$ with $\beta(G)/|G|>1/2$. 
In particular, the set $\{\beta(G)/|G| : G\mbox{  finite group}\}\subset \mathbb{Q}$ has no limit point  strictly between $1/2$ and $1$.


\subsection{The Noether number and its generalization}\label{sec:prel}

Throughout this article  $\field$ is a fixed  algebraically closed base field and
 $G$ is a finite group of order not divisible by  $\mathrm{char} (\field)$, unless explicitly stated otherwise. 
 
By a {\it graded module} we mean an $\mathbb{N}$-graded $\field$-vector space $M=\bigoplus_{d=0}^\infty M_d$, which is a graded module over  a commutative 
$\mathbb{N}$-graded $\field$-algebra $R=\bigoplus_{d=0}^\infty R_d$ such that $R_0= \field$ is the base field when $R$ is unital, 
or $R_0 = \{0\}$ otherwise (in the latter case we still assume that the multiplication map is $\field$-bilinear).  
We set 
 $M_{\ge s} := \bigoplus_{d\ge s} M_d$,   
$M_{\le s} := \bigoplus_{d=0}^s M_d$, and $M_{>s}:=\bigoplus_{d>s}M_d$. 
We also use  the notation $M_+ := M_{ \ge 1}$, so if we regard $R$ as a module over itself, 
its maximal homogeneous ideal is  $R_+$. 
If $M$ is  generated as an $R$-module in bounded degree then set  
\[ \beta(M, R) := \min \{s \in\mathbb{N} :  M \mbox{ is generated as an }R\mbox{-module by }M_{\le s}\} \]
 and write $\beta(M,R) = \infty$ otherwise.
By the graded Nakayama Lemma, a module
$M$ is  generated  by its homogeneous elements $\{m_{\lambda}\mid \lambda\in\Lambda\}$
if and only if 
the $\field$-vector space $M/R_+M$ is spanned by the images $\{\overline{m}_{\lambda}\mid \lambda\in\Lambda\}$. 
As a consequence, $\beta(M,R)$ is the top degree of the factor space $M/R_+M$, inheriting the grading from $M$. 
Here by the top degree of an $\mathbb{N}$-graded vector space we mean the supremum of the degrees of  non-zero homogeneous components (for the zero space the top degree is defined to be zero). 
Obviously we have $\beta(M,R)=\beta(M,R_+)$. 

The subalgebra of $R$ generated by $R_{\le s}$ will be denoted by $\field[R_{\le s}]$. 
For subspaces $S,T$ of an $\field$-algebra $L$ we write $ST$ for the subspace spanned by the products $\{st\mid s\in S,t\in T\}$, 
and use the notation $S^k:=S\dots S$ ($k$ factors)  accordingly.

We  set $\beta(R) := \beta(R_+, R)$. It is zero if $R=R_0$ and otherwise it is the 
supremum of the degrees of homogeneous elements in $R_+ \setminus R_+^2$. 
In other words, $\beta(R)$ is the minimal  $n$ such that $R$ is generated as an $\field$-algebra by homogeneous elements of degree at most $n$ 
when $R$ is generated in bounded degree, and  $\beta(R)=\infty$ when $R$ is not generated in bounded degree.

Let us apply the above concepts in the  more particular setting of invariant theory. 
Here we are given a group $G$ and a finite dimensional $\field$-vector space $V$ equipped with a group homomorphism $G \to \GL(V)$;
in this situation we also say that $V$ is  a (left) $G$-module. 
As an affine space, $V$ has a coordinate ring $\field[V]$ which is defined 
in abstract terms as the symmetric tensor algebra of the dual space $V^*$. 
Thus $\field[V]$ is isomorphic to a polynomial ring in $\dim(V)$ variables,
so in particular it is a graded ring and $\field[V]_1 \cong V^*$. 
The left action of $G$ on $V$ induces a right action on $V^*$ given as  
$x^g(v) = x(gv)$ for any $g\in G, v\in V$ and $x\in V^*$. 
This right action of $G$ on $V^*$ extends multiplicatively onto the whole $\field[V]$. 
Our basic object of  study is the \emph{ring of polynomial invariants} defined as 
\[\field[V]^G := \{ f \in \field[V]: f^g = f \quad \forall g \in G \}.\]
$\beta(G,V) := \beta(\field[V]^G)$ is called the \emph{Noether number} of the $G$-module $V$. 

By a classic result of Hilbert in  \cite{hilbert90}  $\beta(G,V)$ is finite 
if $G$ is \emph{linearly reductive}. When $G$ is finite even more can be said. 
The \emph{global degree bound} for a finite group $G$ is defined as
\begin{align*}\beta(G) := \sup_{V} \beta(G,V) \end{align*}
where $V$ runs through all $G$-modules over the field $\field$. 
By Noether's degree bound  \eqref{noetherbound},
 if $|G|$ is not divisible by $\mathrm{char} (\field)$ then $\beta(G)$ is finite. 
The converse of this statement is also true: it was proved in \cite{derksen-kemper-global} for $\mathrm{char}(\field)=0$
and subsequently in \cite{bryant-kemper} for the whole non-modular case that
the finiteness of $\beta(G)$ implies the finiteness of the group $G$, as well. 
As for the modular case, i.e. when $\mathrm{char}(\field)$ divides $|G|$, Richman constructed in \cite{richman} 
a sequence of $G$-modules $V_1,V_2, ...$ such that $\beta(G,V_i) \to \infty$ as $i \to \infty$, 
so in this case  $\beta(G)$ is not finite.

Note that we suppressed $\field$ from the notation $\beta(G)$. 
The dependence of $\beta(G)$ on the field $\field$ was studied by Knop in \cite{knop}. 
He proved that  $\beta(G)$ is the same for every field $\field$ with the same characteristic. 
In particular this implies that $\beta(G)$ is the same for $\field$ and its algebraic closure. 
So our running assumption that $\field$ is algebraically closed causes no loss of generality in the results.

Now let us summarize  the previously known reduction lemmata by means of which 
$\beta(G)$ can be bounded through  induction on the structure of $G$: 

\begin{lemma} \label{lemma:red}
We have $\beta(G)/|G| \le \beta(K)/|K|$ for any subquotient $K$ of $G$.
\end{lemma}

\begin{proof} 
For any subgroup $H \le G$, resp. for any normal subgroup $N \triangleleft G$ the following reduction lemmata hold: 
\begin{align} 
\beta(G) &\le  [G:H] \beta(H); \\
\beta(G) &\le \beta(G/N) \beta(N).
\end{align}
These were proved for characteristic $0$ by Schmid (see Lemma~3.2 and 3.1 in \cite{schmid})
and subsequently extended to the case when $\mathrm{char}(\field) \nmid |G|$ in \cite{sezer}, \cite{fleischmann:2}, \cite{knop}. 
Our claim follows after dividing by $|G|$ the above inequalities and using that $\beta(N)/|N|\le 1$ by \eqref{noetherbound}.   \end{proof}

We will introduce  here a generalization of the Noether number with the intent of improving and generalizing Schmid's reduction lemmata  above:
For a graded $R$-module $M$ and an integer $k\ge 1$ set 
\[ \beta_k(M, R) : = \beta(M, R_+^k).\]
Note that $\beta_{1}(M,R) = \beta(M,R)$.
The abbreviation $\beta_k(R):= \beta_k(R_+,R)$ will also be used. 
For  a representation $V$ of a finite group $G$  over the field $\field$ we set $\beta_k(G,V):= \beta_k(\field[V]^G)$. 
The trivial bound $\beta_k(G,V) \le k\beta(G,V)$ shows that this quantity is finite.
We also set 
\[\beta_k(G) := \sup\{\beta_k(G,V)\mid V\text{ is a finite dimensional }G\text{-module over }\field\} \] 
suppressing $\field$ from the notation as in the case of $\beta(G)$. 
We shall refer to these numbers as the {\it generalized Noether numbers} of the group $G$.

\subsection{Reduction lemmata}\label{sec:red}

Our starting point is the following alternative characterization of the generalized Noether number: 

\begin{proposition}\label{prop:altnoetnum} 
$\beta_k(G)$ is the minimal positive integer $d$ having the  property that for any 
finitely generated commutative graded $\field$-algebra $L$ (with $L_0=\field$) on which $G$ acts via 
graded $\field$-algebra automorphisms we have 
\[L^G\cap L_+^{d+1}\subseteq 
(L^G_+)^{k+1}.\]
\end{proposition}  

\begin{proof} Let $L$ be a finitely generated commutative graded $\field$-algebra $L$  with $L_0=\field$ on which $G$ acts via graded $\field$-algebra automorphisms.  
Take a finite dimensional $G$-submodule $W\subset L_+$ generating $L$ as an $\field$-algebra, and set $V:=W^*$. Then the $\field$-algebra surjection 
$\pi:\field[V]\to L$ extending the canonical isomorphism $\field[V]_1=W^{**}\cong W\subset L$ 
is $G$-equivariant and maps $\field[V]_+$ onto $L_+$. Moreover, $\pi$ restricts to a surjection 
$\field[V]_+^G\to L_+^G$ by the assumption $\mathrm{char}(\field)\nmid |G|$. 
So we have  
\[L^G\cap L_+^{\beta_k(G)+1}=\pi(\field[V]_{\ge \beta_k(G)+1}^G)
\subseteq \pi((\field[V]^G_+)^{k+1})=(L^G_+)^{k+1}.\] 
For the reverse inequality let $L:=\field[V]$, where $V$ is a finite dimensional $G$-module with 
$\beta_k(G,V)=\beta_k(G)$. 
\end{proof} 

\begin{lemma} \label{lemma:myred}
Let $N$ be a normal subgroup of $G$. Then for any $G$-module $V$ we have
\[ \beta_k(G,V) \le \beta_{\beta_k(G/N)}(N,V) \]
Consequently the inequality $\beta_k(G) \le \beta_{\beta_k(G/N)}(N)$ holds, as well.
\end{lemma}

\begin{proof} 
We shall apply Proposition~\ref{prop:altnoetnum}  for the algebra $L:=\field[V]^N$;  
denote $R:=\field[V]^G$. The subalgebra $L$ of $\field[V]$ is $G$-stable, and the action of $G$ on $L$  factors through $G/N$, and $R=L^{G/N}$. 
Setting $s:=\beta_{\beta_k(G/N)}(N,V)$, 
we have 
\begin{align*}
R_{\ge s+1}= R\cap  L_{\ge s+1} &\subseteq 
L^{G/N}\cap L_+^{\beta_k(G/N)+1}
 \subseteq (L_+^{G/N})^{k+1}=(R_+)^{k+1}. \qedhere \end{align*}
\end{proof}

A weaker version of Lemma~\ref{lemma:myred} remains true for any subgroup $H \le G$ which is not necessarily normal.
To show this we will make use of the following relativized version of the Reynolds operator (see e.g.  \cite{neusel} p. 33):
Let $H \le G$ be a subgroup 
and $g_1, ..., g_n$ a system of right coset representatives of $H$.
For a $G$-module $V$ the map $\tau_H^G: \field[V]^H \to\field[V]^{G}$ called the \emph{relative transfer map}
 is defined by the sum 
\[ \tau_H^G(u) = \sum_{i=1}^n u^{g_i} .\] 
In the special case when $H$ is the trivial subgroup $\{1_G\}$, we recover the \emph{transfer map} 
$\tau^G:\field[V]\to\field[V]^G$. 
If $\mathrm{char}(\field)$ does not divide $[G:H]$ then $\tau_H^G$ 
is a graded $\field[V]^G$-module epimorphism from $\field[V]^H$ onto $\field[V]^G$.  We shall use this fact most frequently in the following form: 

\begin{proposition} \label{trans}
If $\mathrm{char}(\field)$ does not divide $[G:H]$, then we have 
$\beta_k(G,V)\le \beta_k(\field[V]^H_+, \field[V]^G)$. 
\end{proposition}

\begin{proposition}\label{prop:knopeta} 
Let $J$ be a non-unitary commutative $\field$-algebra on which a finite group $G$ acts via $\field$-algebra automorphisms
and let $H \le G$ be a subgroup for which one of the following conditions holds: 
\begin{itemize}
\item[(i)] $\mathrm{char}(\field)>[G:H]$ or $\mathrm{char}(\field)=0$;  
\item[(ii)] $H$ is normal in $G$ and $\mathrm{char}(\field)$ does not divide $[G:H]$; 
\item[(iii)] $\mathrm{char}(\field)$ does not divide $|G|$. 
\end{itemize}
Then we have
\[  (J^H)^{[G:H]} \subseteq  J^H  J^G + J^G  \] 
\end{proposition}

\begin{proof} (i) 
Let $f\in J^H$ be arbitrary and $\mathcal{S}$ a system of right $H$-coset  representatives in $G$.
Then $f$ is a root of the monic polynomial $\prod_{g\in \mathcal{S}}(t- f^g)\in J[t]$.
Obviously all coefficients of this polynomial are $G$-invariant. 
Consequently, $f^{[G:H]}\in J^H J^G+J^G$ holds for all $f\in J^H$. 
Take arbitrary  $f_1,\ldots,f_r\in J^H$ where $r=[G:H]$. Then the product 
$r!f_1\cdots f_r$ can be written as an alternating sum of $r$th powers of sums of subsets of $
\{f_1,\ldots,f_r\}$ (see e.g. Lemma 1.5.1 in \cite{benson}), hence $f_1\cdots f_r\in J^H J^G+ J^G$. 

(ii) (This  is a variant of a result of Knop,  Theorem 2.1 in \cite{knop}; the idea appears in Benson's simplification of Fogarty's argument from \cite{fogarty}, see Lemma 3.8.1 in \cite{derksen-kemper}). 
Let $\mathcal{S}$ be a system of $H$-coset  representatives in $G$. 
For each $x \in \mathcal{S}$ choose an arbitrary element $a_x \in J^H$. 
It is easily checked that 
\begin{align}\label{eq:knop} 0 \; =\; \sum_{y \in \mathcal{S}} \prod_{x \in \mathcal{S}} (a_x- a_x^{x^{-1}y}) 
&=  \sum_{U \subseteq \mathcal{S}} (-1)^{|U|} \delta_U \qquad \text{where}\\ \notag
 \delta_U &:= \prod_{x \not\in U}a_x  \sum_{y \in \mathcal{S}}(\prod_{x \in U} a_x^{x^{-1}})^y. 
 \end{align}
Note that $a_x^g\in J^H$ for all $x\in \mathcal{S}$ and $g\in G$ by normality of $H$ in $G$.  Therefore 
 $\delta_U=   \prod_{x \not\in U} a_x \; \tau^G_H\left(\prod_{x \in U} a_x^{x^{-1}}\right)$.
Thus $\delta_{\mathcal{S}}\in J^G$ and $\delta_U \in J^H J^G $ for every $U \subsetneq \mathcal{S}$, except for $U = \emptyset$, 
when we get the term $[G:H] \prod_{x \in \mathcal{S} } a_x$. 
Given that $[G:H] \in \field^{\times}$  and  the elements $a_x$ were arbitrary the claim  follows.

(iii) Let $\mathcal{S}$ be a system of left $H$-coset  representatives in $G$. 
Apply the transfer map $\tau^H:J\to J^H$ to the equality \eqref{eq:knop}, and observe that 
\begin{align}
\tau^H(\delta_U)=\prod_{x \not\in U}a_x \sum_{h\in H} \sum_{y \in \mathcal{S}}(\prod_{x \in U} a_x^{x^{-1}})^{yh}= \prod_{x \not\in U}a_x \tau^G(\prod_{x \in U} a_x^{x^{-1}}).
\end{align}
This shows that $\tau^H(\delta_U)\in J^HJ^G+J^G$ for all non-empty subsets $U\subseteq \mathcal{S}$, and 
$\tau^H(\delta_{\emptyset})=|G|\prod_{x\in\mathcal{S}}a_x$, implying the claim as in (ii). 
\end{proof}

\begin{remark}
Finiteness of $G$ can be replaced by  finiteness of $[G:H]$  in (i) and (ii) above. 
\end{remark}

\begin{corollary}\label{cor:G:H+1} 
Keeping the assumptions of Proposition~\ref{prop:knopeta} on $G$, $H$ and $\mathrm{char}(\field)$, 
let $V$ be a $G$-module, $I := \field[V]^H$, $R:= \field[V]^G$. Then for any  
graded $I$-module $M$ we have
\begin{align} \beta_k(M,R)\le \beta_{k[G:H]}(M,I). \end{align}
In particular we have the inequality 
\begin{align}  
 \beta_k(G,V)  & \le \beta_{k[G:H]} (H,V).\end{align}
\end{corollary} 

\begin{proof} 
Apply Proposition~\ref{prop:knopeta} for $J:=\field[V]_+$. Then $J^H=I_+$ and $J^G=R_+$, so $I_+^{[G:H]}\subseteq I_+R_++R_+$. 
Consequently $(I_+^{[G:H]})^k\subseteq I_+R_+^k+R_+^k$, hence 
$MI_+^{k[G:H]}\subseteq MR_+^k$. Thus the top degree of the factor space $M/MR_+^k$ is bounded by the top degree of $M/MI_+^{k[G:H]}$, 
implying the first inequality.
For the second note that $\beta_k(G,V) = \beta_k(R) \le \beta_k(I_+,R)$  by Proposition~\ref{trans} 
and $\beta_k(I_+,R)\le  \beta_{k[G:H]}(I_+,I) = \beta_{k[G:H]}(H, V)$.
\end{proof}

\begin{remark}\label{remark:babynoether}
 It is conjectured that $\beta(G,V)\leq [G:H]\beta(H,V)$ holds in fact  whenever $\mathrm{char}(\field)\nmid [G:H]$.
This  open question is mentioned under the name  ``baby Noether gap"  in Remark 3.8.5 (b) in \cite{derksen-kemper} or on page 1222 in \cite{kemper-separating}.  
\end{remark}

Finally we present some rather technical results which will be used later in Chapter~\ref{ch:semidir} to obtain upper bounds on $\beta(G)$:

\begin{lemma} \label{lemma:induk1}
Let $M$ be a graded module over a graded ring $I$,  
and $S \subseteq I$ a graded subalgebra. 
Then for any integers $k > r \ge 1$ we have
\begin{align*} \label{eq:induk}
\beta_{k}(M,I) \le   \max \{ \beta(M,S) + \beta_{k-r-1}(S) , \beta_{r}(M,I) + \beta_{k-r}(S) \}
\end{align*}
\end{lemma} 

\begin{proof} 
Assume that $d\in \mathbb{N}$ is greater than the right hand side of this inequality. 
Then 
\begin{equation}\label{eq:rev1}M_d \subseteq  M_{\le \beta(M,S)}S_{> \beta_{k-r-1}(S)} \subseteq MS_+^{k-r}. \end{equation} 
Note that for any positive integer $j$ the top degree of $S_+^j/S_+^{j+1}$ is trivially bounded by the top degree of the larger space 
$S_+/S_+^{j+1}$. In other words $\beta(S_+^j,S)\le\beta(S_+,S_+^j)=\beta_j(S)$, thus 
$MS_+^j\subseteq M (S_+^j)_{\le \beta_j(S)}$. It follows that 
\begin{equation}\label{eq:rev2} MS_+^{k-r} = M(S_+^{k-r})_{\le \beta_{k-r}(S)}.\end{equation} 
Combining \eqref{eq:rev1}, \eqref{eq:rev2} with the assumption $d>\beta_r(M,I)+\beta_{k-r}(S)$ we get 
\[M_d  \subseteq M_{>\beta_{r}(M,I)}  S_+^{k-r} \subseteq M I_+^{r} S_+^{k-r} \subseteq MI_+^{k}.\]  
This proves that $d>\beta_k(M,I)$.
\end{proof} 

\begin{lemma}\label{induk} 
For a $G$-module $V$ and subgroup $H \le G$ as in Proposition~\ref{prop:knopeta}
set $L := \field[V]$, $M:=L_+/L_+^GL_+$. For any $1\le r< [G:H]$ and   $s\ge 1$  we have 
\begin{align*} 
\beta(L_+,L^G) \le ([G:H] -r)s + \max \{\beta_r(M, L^H) ,  \beta(M, \field[L^H_{\le s}]) -s \}
\end{align*} 
\end{lemma}

\begin{proof} We have $ \beta(L_+,L^G)=\beta(M,L^G)\le\beta_{[G:H]}(M,L^H)$ by Corollary~\ref{cor:G:H+1}. 
Applying Lemma~\ref{lemma:induk1} with $k:=[G:H]$,  $I:=L^H$,  $S := \field[I_{\le s}]$ and noting that 
$\beta_k(S) \le ks$ we obtain the above inequality. 
\end{proof} 

\begin{remark}
(i) A version of  Lemma~\ref{induk} limited to the abelian case appears  in \cite{geroldinger-halterkoch} as Lemma 6.1.3.

(ii)
The use of  Lemma~\ref{lemma:myred} and Corollary~\ref{cor:G:H+1} on the generalized Noether number stems from the fact  that 
 for $k>1$ the number $\beta_k(G,V)$ in general is strictly smaller than $k\beta(G,V)$, 
 as it can be seen in Section~\ref{sec:davenport} already for abelian groups.  See also \cite{CzD:3} for more information in this respect. 
\end{remark}


\subsection{The Davenport constant}\label{sec:davenport} 

 A \emph{character} of  an abelian group $A$ is a group homomorphism from $A$ to the multiplicative group $\field^{\times }$ of the base field. 
The set of characters of $A$ is denoted by $\hat{A}$; it is naturally an abelian group, and in fact there is a (non-canonical) isomorphism $\hat{A} \cong A$. 
Let $V$ be a representation of $A$ over the  base field $\field$. 
Since $\field$ is algebraically closed and $\mathrm{char}(\field)$ does not divide 
$|A|$ by our conventions, 
$V$ decomposes as a direct sum of $1$-dimensional representations.
This means that $V^*$ has an $A$-eigenbasis $\{x_{1},...,x_{n}\}$. 
The character 
$\theta_i \in \hat{A}$ given by $x_i^a = \theta_i(a) x_i$ is called the \emph{weight} of $x_i$.
We shall always tacitly choose such an $A$-eigenbasis as the variables in the polynomial algebra $\field[V]=\field[x_{1}, ..., x_{n}]$. 
Let $M(V)$ denote the set of monomials in $\field[V]$; this is a monoid with respect to ordinary multiplication and unit element $1$. 
On the other hand we denote by $\mathcal{M}(\hat{A})$ the free commutative monoid generated by the elements of $\hat{A}$.
Due to our choice of variables in $\field[V]$ we can define a monoid homomorphism $\Phi: M(V) \to \mathcal{M}(\hat{A})$
by sending each variable $x_i$ to its weight $\theta_i$.  
We shall call $\Phi(m)$ the \emph{weight sequence} of the monomial $m \in M(V)$. 
We prefer to write $\hat{A}$ additively, hence 
for any character $\chi\in \hat A$ we denote by $-\chi$ the character $a\mapsto \chi(a)^{-1}$, 
$a\in A$.

An element $S \in \mathcal{M}(\hat{A})$ can  be interpreted as a \emph{sequence} 
$S:=(s_1,\ldots,s_n)$ of elements of $\hat{A}$ where repetition of elements is allowed and their order is disregarded. 
The \emph{length} of $S$ is $|S|:=n$. 
By a \emph{subsequence} of $S$ we mean $S_J := (s_j\mid j\in J)$ for some subset $J\subseteq \{1,\ldots,n\}$. %
Given a sequence $R$ over an abelian group $A$ we write 
$R=R_1R_2$ if $R$ is the concatenation of its subsequences $R_1$, $R_2$,
and we call the expression $R_1R_2$ a \emph{factorization} of $R$. 
Given an element $a\in A$ and a positive integer $r$, write $(a^r)$ for the sequence in which $a$ occurs with multiplicity $r$. 
For an automorphism $b$ of $A$ and a sequence $S=(s_1,\dots,s_n)$ we write $S^b$ for the sequence 
$(s_1^b,\dots,s_n^b)$, and we say that the sequences $S$ and $T$ are \emph{similar} if $T=S^b$ 
for some $b \in \Aut(A)$.

Let $\sigma: \mathcal{M}(\hat{A}) \to \hat{A}$ be the monoid homomorphism 
which assigns to each sequence over $A$ the sum of its elements. 
The value $\sigma(\Phi(m)) \in \hat A$ is called the \emph{weight of the monomial} $m \in M(V)$ and it will be abbreviated by $\theta(m)$. 
In particular, $\theta(x_i)=\theta_i$ with the notation in the first paragraph of this section. 
The kernel of $\sigma$ is called the \emph{block monoid} of $\hat{A}$, denoted by $\mathcal{B}(\hat{A})$,
and its elements are called zero-sum sequences. Our interest in zero-sum sequences 
and the related results in additive number theory  stems from the observation that
the invariant ring $\field[V]^A$ is spanned as a vector space by all those monomials
for which $\Phi(m)$ is a zero-sum sequence over $\hat A$. Moreover, as an algebra, $\field[V]^A$ is minimally 
generated by those monomials $m$ for which $\Phi(m)$ 
does not contain any proper zero-sum subsequences. 
These are  called \emph{irreducible} zero-sum sequences, 
and they form the Hilbert basis of the monoid $\mathcal{B}(\hat A)$. 
A sequence is \emph{zero-sum free} if it has no non-empty zero-sum subsequence.

The \emph{Davenport constant} $\davenport(A)$ of  $A$ is defined as the length of 
the longest irreducible zero-sum sequence over $A$.  
It is an extensively studied quantity,  see for example \cite{geroldinger-gao}. 
As it is seen from our discussion:
\begin{equation}\label{eq:davenportbeta}
\davenport(A)=\beta(A).\end{equation}

The {\it generalized Davenport constant} $\davenport_k(A)$ is introduced in \cite{halter-koch} as the length of the longest zero-sum sequence 
that cannot be factored  into more than $k$ non-empty zero-sum sequences. 
It is evident from the above discussion that $\davenport_k(A) = \beta_k(A)$.  Moreover
Lemma~\ref{lemma:myred}  applied to abelian groups yields for any subgroup $B\leq A$ that:
 \begin{align} 
 \davenport_k(A) &\leq \davenport_{\davenport_k(A/B)}(B); \\
 \davenport_k(A) &\le \davenport_{\davenport_k(B)} (A/B). 
 \end{align} 
The second inequality follows from the first because
$A$ has a subgroup $C \cong A/B$ for which $A /C \cong B$, 
hence the role of $A/B$ and $B$ can be reversed in this formula. 
This inequality appears as Proposition 2.6 in \cite{delorme}.  

For the cyclic group $Z_n$ we have $\davenport_k (Z_n) = kn$. 
We close this section with two more results on $\davenport_k$ which will be used later on.

\begin{proposition}[Halter-Koch, \cite{halter-koch} Proposition 5] \label{prop:halter-koch} 
For any $n \mid m $ we have
\[ \davenport_k (Z_n \times Z_m) = km + n -1.  \]
\end{proposition}

\begin{proposition}[Delorme-Ordaz-Quiroz, \cite{delorme} Lemma 3.7] \label{prop:Z2Z2Z2} 
\[ \davenport_k(Z_2 \times Z_2 \times Z_2) =
\begin{cases}
4 & \text{ if } k =1; \\
2k + 3 & \text{ if } k>1. 
\end{cases}
\]
\end{proposition}


\section{The semidirect product }\label{ch:semidir}

Our main aim in the present chapter is to give upper bounds  on $\beta(Z_p\rtimes Z_q)$ for the non-abelian semidirect product $Z_p\rtimes Z_q$, where $p,q$ are odd primes, $q\mid p-1$.  
It is an open conjecture of Pawale reported in \cite{wehlau} that  $\beta(Z_p\rtimes Z_q)=p+q-1$. 
The lower bound $\beta(Z_p \rtimes Z_q) \ge p+q -1$ follows from a more general result in \cite{CzD:2} 
(and can also be seen directly).  
We provide here upper bounds that improve on  \cite{domokos-hegedus} and \cite{pawale}, and are sufficient for the proof of Theorem~\ref{thm:main}. 

\subsection{Extending Goebel's algorithm}\label{sec:goebel}

Let $G$ be a finite group with  a proper  abelian normal subgroup $A$.
Consider a monomial representation $G \to \GL(V)$ which maps  $A$ to diagonal matrices. 
This presupposes the choice of a basis  $x_1,...,x_n$ in the dual  space $V^*$, 
which are $A$-eigenvectors permuted up to scalars under the action of $G/A$. 
We shall always tacitly choose  them as the variables in the coordinate ring $L:=\field[V]$. 
Goebel  developed an algorithm for the case when $V$ is a permutation representation
(see \cite{goebel}, \cite{neusel}, \cite{derksen-kemper}) 
which we will adapt here to this more general case. 

The conjugation action of $G$ on $A$ induces an action on $\hat A$ in the standard way, and we consider the corresponding action of $G$ on
$\mathcal{M}(\hat A)$. 
Extending slightly  the notation of Section~\ref{sec:davenport} we define the weight sequence and the weight for any non-zero scalar multiple of a monomial: 
for $m\in M(V)$ and $c\in\field^\times$ set $\Phi(cm):=\Phi(m)$ and $\theta(cm):=\theta(m)$.  
It is easy to check that for any monomial $m\in M(V)$ and $g\in G$ we have $\Phi(m^g)=\Phi(m)^g$ and consequently $\theta(m^g)=\theta(m)^g$. 
Enumerate the $G$-orbits in $\hat A$ in a fixed order $O_1,\dots,O_l$. 
For a $G$-orbit $O$ in $\hat A$ let $S^O$ be the subsequence of $S$ consisting of its elements belonging to $O$.  
Now $S$ has the canonic factorization $S=S^{O_1}\dots S^{O_l}$. 
In addition any sequence $S$ over $\hat A$ has 
a unique factorization $S = R_1R_2... R_h$ such that each $R_i \subseteq \hat A$ is multiplicity-free and $R_1 \supseteq ... \supseteq R_h$; 
we call this the \emph{row decomposition} of $S$ and we refer to  $R_i$ as the $i$th  \emph{row}  of $S$, whereas  $\supp(S) := R_1$ is its \emph{support} and $h(S):=h$ is its \emph{height}. In other terms  $h(S)$  is the maximal multiplicity of the elements in $S$. 
The intuition behind this  is that we like to think of sequences as Young diagrams where
the multiplicities  in $S$ of the different elements of $\hat A$  are represented by the heights of the columns. 
Denote by $\mu(S)$ the non-increasing sequence of integers $(\mu_1(S),\dots,\mu_h(S)):=(|R_1|, ..., |R_h|)$. 
By the  \emph{shape} $\lambda(S)$ of $S$ we mean the $l$-tuple of such partitions 
\[ \lambda(S) := (\mu(S^{O_1}),\dots,\mu(S^{O_l})). \]
The set of the shapes is equipped with the usual reverse lexicographic order, i.e. $\lambda(S) \prec \lambda(T)$ if $\lambda(S)\neq\lambda(T)$ and 
for the  smallest index $i$ such that $\mu(S^{O_i})\neq \mu(T^{O_i})$,  we have $\mu_j(S^{O_i})>\mu_j(T^{O_i})$ for the smallest index $j$ with $\mu_j(S^{O_i})\neq \mu_j(T^{O_i})$. 
Observe that $\lambda(ST) \prec \lambda(S)$ always holds
but on the other hand $\lambda(S) \prec \lambda(S')$ does not imply $\lambda(ST) \prec \lambda(S'T)$. 
Abusing notation for any monomial $m \in \field[V]$ we write $\lambda(m)$, $h(m)$ and $\supp(m)$ 
for the shape, height and the support of its weight sequence $\Phi(m)$.

In the following we shall assume that we fixed a subset $\mathcal{V}$ of the variables permuted by $G$ up to non-zero scalar multiples; we adopt the convention that unless 
explicitly stated otherwise, $\mathcal{V}$ is the set of all variables. Any monomial $m$ factors as $m=m_{\mathcal{V}}m_{\widehat{\mathcal{V}}}$, where $m_{\mathcal{V}}$ is a product of variables belonging to $\mathcal{V}$, and $m_{\widehat{\mathcal{V}}}$ does not involve variables from $\mathcal{V}$. We shall also use  the notation 
$\lambda_{\mathcal{V}}(m):=\lambda(m_{\mathcal{V}})$. 

\begin{definition}~\label{def:terminal} An $A$-invariant monomial $u$ is a \emph{good factor} of a monomial $m=uv$  if   
$\lambda_{\mathcal{V}}(u^bv)\prec \lambda_{\mathcal{V}}(m)$ 
holds for all $b\in G\setminus A$; note that this forces $0<\deg(u)<\deg(m)$. 
We say that $m$ is \emph{terminal} if it has no good factor. 
\end{definition}

\begin{lemma} \label{goebel} $L_+=\field[V]_+$ is generated as an $L^G$-module by the terminal monomials. 
\end{lemma} 

\begin{proof} We prove by induction on $\lambda_{\mathcal{V}}(m)$ with respect to $\prec$ that if $m$ is not terminal, then it can be expressed 
modulo $L_+L^G_+$
as a linear combination of terminal monomials. Indeed, take a good divisor $u$ of $m=uv$. 
Then we have
\begin{equation}\label{eq:tekeres} 
\sum_{b\in G/A}u^bv=\tau^G_A(u)v\in L_+^GL_+. \end{equation} 
Since for every monomial in the sum on the left hand side except for $m$ we have $\lambda_{\mathcal{V}}(u^bv)\prec\lambda_{\mathcal{V}}(m)$, our claim on $m$ holds by the induction hypothesis. 
\end{proof} 

At this level of generality there might be an element $b \in G\setminus A$ such that $\theta(x_i^b) = \theta(x_i)$ for every variable $x_i$, 
and then  every monomial  qualifies as terminal by our definition. 
The concept of terminality is particularly useful when 
\begin{align} \label{G/A-kikotes}
\chi^b\neq \chi \quad \text{ for each  } \quad b\in G\setminus A\quad \text{ and  }\quad \chi\in\hat A\setminus\{0\}.
\end{align} 
For the rest of this section we assume that \eqref{G/A-kikotes} holds for $(G,A)$. An obvious necessary condition for \eqref{G/A-kikotes} to hold is that $A$ must be a self-centralizing, 
hence maximal abelian subgroup in $G$, and the order of $G/A$ must divide $|A|-1$, hence $G$ is the semidirect product of $A$ and $G/A$
by the Schur-Zassenhaus theorem. 
In fact  condition \eqref{G/A-kikotes} is equivalent to the requirement that  
$G$ is a Frobenius group with abelian  Frobenius kernel $A$. 
In this article we will only study  in greater detail 
the non-abelian semidirect products $Z_p\rtimes Z_q$, $Z_p\rtimes Z_{q^n}$ where $Z_{q^n}$ acts faithfully on $Z_p$, and the alternating group $A_4$.

Note that if \eqref{G/A-kikotes} holds, then for any non-trivial $1$-dimensional $A$-module $U$ the $G$-module $\Ind_A^G(U)$ is irreducible
by Mackey's irreducibility criterion (cf. \cite{serre} ch. 7.4).
Moreover, the set of $A$-characters occurring in $\Ind_A^G(U)$ coincides with the $G/A$-orbit of the character of $A$ on $U$, and each $A$-character occurring in $\Ind_A^G(U)$ has multiplicity one.
Hence the $G/A$-orbits in  $\hat A\setminus \{0\}$ are in bijection with  the isomorphism classes of those irreducible $G$-modules that are induced from a $1$-dimensional $A$-module.

\begin{definition}
A monomial $m \in \field[V]$ or its weight sequence $S= \Phi(m)$ is called a \emph{brick} if 
$S$ is the orbit of a minimal non-trivial subgroup of $G/A$. 
\end{definition}

\begin{remark} 
(i) If \eqref{G/A-kikotes} holds then every brick is $A$-invariant. 
Indeed, when $m \in \field[V]$ is a brick then $\Phi(m)$ is stabilized by some non-identity element $b \in G/A$,
hence $\theta(m)$ is fixed by $b$, which is only possible by \eqref{G/A-kikotes} if $\theta(m)=0$. 

(ii) If a monomial $m$ is not divisible by a brick, then $\Phi(m)\neq \Phi(m^b)$ for each $b\in G\setminus A$. 
\end{remark}

\begin{definition} \label{def:gapless}
A sequence $S$ over $\hat A$ with row-decomposition $S=R_1...R_h$ is called 
\emph{gapless} if for all $G/A$-orbits $O$ and all $i<h$ such that $R_i\cap O\neq \emptyset$
we have $R_i\cap O\neq R_{i+1}\cap O$ 
or $R_i \cap O = R_{i+1} \cap O =O$. 
A monomial $m\in \field[V]$ is called \emph{gapless} if its weight sequence $\Phi(m)$ is gapless. 
 \end{definition}

For our next result we will need the following easy combinatorial fact:

\begin{lemma} \label{lemma:easy} 
For any sequence $S = (s_1, ..., s_d)$ over an abelian group $A$ let 
$\Sigma(S) := \{ \sum_{i\in I} s_i: I  \subseteq \{1,...,d \} \}$.
If $A= Z_p$ for a prime $p$ and $S=(s_1,...,s_d)$ a sequence of non-zero elements of $Z_p$ 
then \[ |\Sigma(S)| \ge \min \{ p, d+1 \}.\] 
\end{lemma}

\begin{proof} 
By the Cauchy-Davenport Theorem $ |A+B|\ge \min\{p,|A|+|B|-1\} $
for any non-empty subsets $A,B \subseteq Z_p$. Our claim follows from this by induction 
on $d$, as
$|\Sigma(S)| \ge |\Sigma(s_1,...,s_{d-1})| + |\{0, s_d \}| -1 \ge d+2-1$ for any $1<d<p$,
while the case $d=1$ is trivial.
\end{proof}

\begin{proposition} \label{prop:gapless}
Let $G=A\rtimes Z_n$ where  $A \cong Z_p$ for some prime $p$ and $Z_n$ acts faithfully on $A$. 
Let $V$ be a $G$-module and   $L:= \field[V]$, $R:= \field[V]^G$, and $\mathcal{V}$ any  subset of the  variables permuted by $G$ up to non-zero scalar multiples. 
Then $L_+/L_+R_+$ is spanned by monomials of the form $b_1\dots b_rm$, where 
each $b_i$ is an $A$-invariant variable or a  brick composed of variables in $\mathcal{V}$ 
 while  $m_{\mathcal{V}}$ has a gapless divisor of degree 
at least 
\[\min\{\deg(m_{\mathcal{V}}), \deg(m)-p+1\}.\] 
\end{proposition}

\begin{proof}  Since $A$ has prime order, a non-trivial character $\chi\in\hat A$ takes distinct values on the elements of $A$. As $Z_n$ acts faithfully on $A$, for any non-identity element $g$ of $Z_n$ there is an $a\in A$ with $a^g\neq a$, thus $\chi^g(a)=\chi(a^g)\neq \chi(a)$. So \eqref{G/A-kikotes} holds for $(G,A)$. 
By Lemma \ref{goebel} it suffices to show that for any terminal monomial  $m\in L_+$  
not containing a brick over $\mathcal{V}$ or  an $A$-invariant variable, 
$m_{\mathcal{V}}$ has  a gapless divisor of degree at least $\min\{\deg(m_{\mathcal{V}}), \deg(m)-p+1\}$. 
Let $m^*$ be a gapless divisor of $m_{\mathcal{V}}$ of maximal possible degree, and suppose for contradiction that $\deg(m^*)<\min\{\deg(m_{\mathcal{V}}), \deg(m)-p+1\}$.  Then there is a variable $x$ such that 
$m^*x$ is a divisor of $m_{\mathcal{V}}$ and $m^*x$ is not gapless, moreover, the index of the orbit $O_i$ containing $\theta(x)$ is minimal possible, i.e. for all $j<i$ we have  $\Phi(m^*)^{O_j}=\Phi(m_{\mathcal{V}})^{O_j}$. 
Let $\Phi(m^*)^{O_i} = R_1R_2 ... R_h$ be the row decomposition of $\Phi(m^*)^{O_i}$, and denote by $t$  the multiplicity of $\theta(x)$ in $\Phi(m^*)$. 
It is then necessary that   $R_t = R_{t+1} \cup \{ \theta(x) \}$, for otherwise $m^*x$ would still be gapless. 
Take a divisor $u \mid m^*$ with $\Phi(u) = R_{t+1}$, hence $\Phi(ux)=R_t$ and the row decomposition of $m^*/u$ is $R_1\dots R_tR_{t+2}\dots R_h$.  
Now consider the remainder $m/(m^*x)$: it  contains no variables of weight $0$, 
and its degree is at least $p-1$ by assumption, hence $|\Sigma(\Phi(m/(m^*x)))| = p$ by Lemma~\ref{lemma:easy}.
Thus $m/(m^*x)$ has a (possibly trivial) divisor $\hat u$ for which $\theta(\hat{u}) = - \theta(ux)$. It is easy to see that  $w:=xu\hat u$ is a good divisor of $m$. Indeed,   
set $v := m/w$, and take $b\in G\setminus A$; clearly, $m^*/u$ divides $v$.  For $j<i$, we have $\Phi((w^bv)_{\mathcal{V}})^{O_j}
=\Phi(m_{\mathcal{V}})^{O_j}$. 
Moreover,  
$\mu_s(\Phi((w^bv)_{\mathcal{V}})^{O_i})\ge 
\mu_s(\Phi(m_{\mathcal{V}})^{O_i})$ for $s=1,\dots t$. Here we have  strict inequality at least for one $s$:  
by our assumption  $\Phi((ux)_{\mathcal{V}})=R_t$ is not divisible by a brick, so $R_t^b \setminus R_t \neq \emptyset$, hence 
the support of $\Phi(w^b_{\mathcal{V}})^{O_i}$ is not contained in $R_t$, implying  
$\sum_{s=1}^t\mu_s(\Phi((w^bv)_{\mathcal{V}})^{O_i})>\sum_{s=1}^t \mu_s(\Phi((m^* /u)_{\mathcal{V}})^{O_i})$. 
This contradicts  the assumption that $m$ was terminal. 
\end{proof}


\subsection{Factorizations of gapless monomials}

Denote by $\mathcal{B}$ the ideal of $L=\field[V]$ generated by the bricks, 
and denote by  $\mathcal{G}_d$ the ideal of $L$ generated by the gapless monomials of degree at least $d$. 
Moreover, for a set $\mathcal{V}$ of variables as in Proposition~\ref{prop:gapless}, 
denote by $\mathcal{G}_d(\mathcal{V})$ the ideal of $L$  spanned by monomials with a gapless divisor 
of degree at least $d$ composed from variables in $\mathcal{V}$.

\begin{proposition} \label{binom}
Let $V= \Ind_A^G U $ be  an isotypic $G$-module belonging to a $G$-orbit $O \subseteq \hat A$, 
and  $s$ the index of a minimal nontrivial subgroup of $G/A$. Then
\[ \mathcal{G}_ {d}  \subseteq \mathcal{B}  \qquad \text{ where } \quad d = \binom { |O| - s +1}{2} +1. \]
\end{proposition}

\begin{proof}
Let $m \in \field[V]$ be a gapless monomial not divisible by a brick.  
In the row decomposition $\Phi(m) = R_1 ...R_h$ we then have $|R_{i+1}| < |R_i| $ for every $1 \le i < h$, 
and $|R_1| \le |O| - s$, so $\deg(m)\le 1+2 +...+ (|O|-s) = \binom{ |O| - s +1}{2} $.
\end{proof}

\begin{corollary}\label{cor:skatulya} 
Let $A= Z_p$ and $G = A \rtimes Z_{q^n}$ where $Z_{q^n}$ acts faithfully on $A$. 
Setting $c = \frac{p-1}{q^n} $  and $ d = \binom {q^n - q^{n-1} + 1}{2}$ and $L = \field[W]$, $R = \field[W]^G$ for  a $G$-module $W$   we have
\[ \beta(L_+,R) \le (q^{n}- 2)q + \max \{cd, p+d -1,p+q \}.  \]
\end{corollary}

\begin{proof}  By Lemma~\ref{induk} (applied with $s=q$ and $r=1$) we have 
$\beta(L_+,R)\le (q^n-1)q+\max\{p,\beta(L_+/R_+L_+,S)-q\}$, where 
$S:=\field[I_{\le q}]$.  Apart from $O_0:=\{ 0 \}$,  $Z_p$ contains $c$ different $Z_{q^n}$-orbits  $O_1,\dots,O_c$, each of cardinality $ q^n$, and  
the bricks different  from $O_0$ are all of size  $q$. 
Thus $\beta(L_+/R_+L_+,S)\le\beta(L_+/L_+R_+,\mathcal{B})$, and it is sufficient to show that 
for $e:=\max\{cd+1,p+d,p+q+1\}$, $L_{\ge e}\subseteq L_+R_++\mathcal{B}$. 

Denote by $M^{(i)}$ (resp. $M^{(0)}$) the subspace of $L_{\ge e}$ spanned  by monomials $u$ with $|\Phi(u)^{O_i}|>d$ (resp. $|\Phi(u)^{O_0}|\ge 1$). Clearly $L_{\ge e}\subseteq \sum_{i=0}^cM^{(i)}$. The $A$-invariant variables are bricks, so $M^{(0)}\subseteq\mathcal{B}$. 
Apply Proposition~\ref{prop:gapless} with $\mathcal{V}$ the set of variables of weight in $O_i$ for some fixed $i\in\{1,\dots,c\}$. We obtain that the subspace  $M^{(i)}$ 
is contained in  $R_+L_++\mathcal{B}+\mathcal{G}_{d+1}(\mathcal{V})$. 
By Proposition~\ref{binom}, $\mathcal{G}_{d+1}(\mathcal{V})\subseteq\mathcal{B}$, showing that $M^{(i)}\subseteq R_+L_++\mathcal{B}$. This holds for all $i$, 
hence $L_{\ge e}\subseteq L_+R_++\mathcal{B}$. 
\end{proof}

For the rest of this section let $G$ be the non-abelian semidirect product $Z_p \rtimes Z_q$, where $p,q$ are odd primes and $q \mid p-1$. 
We set  $L:= \field[W]$, $I = \field[W]^{Z_p}$, $R = \field[W]^G$ for an arbitrary $G$-module $W$ and denote by $A$ the normal subgroup $Z_p$ in $G$. 
In this case the bricks are the monomials $m$ with $\Phi(m)=O_i$ for some $i=0,1,\dots,\frac{p-1}q$, 
so a brick is either an $A$-invariant variable or has degree $q$. Moreover, multiplying a gapless monomial by a brick we get a  gapless monomial. 
Thus in the statement of Proposition~\ref{prop:gapless} all the $b_i$ may be assumed to be $A$-invariant variables. 
We need the following  consequence of the Cauchy-Davenport Theorem (see Theorem 5.7.3 in \cite{geroldinger-halterkoch} for a more general statement): 
\begin{lemma} \label{Zp_corner} 
Let $S$ be a sequence over $Z_p$  with maximal multiplicity $h$. 
If $|S| \ge p$ then $S$ has  a zero-sum subsequence $T \subseteq S$  of  length $|T| \le h$. 
\end{lemma}

\begin{corollary} \label{cor:tuske} We have the inequality 
\[ \beta(L_+, R) \le p + \frac{q(q-1)^2}{2}. \]
\end{corollary}
\begin{proof} 
Applying Lemma~\ref{induk} with $r=1$ and $s:=\binom{q}{2}$, and using $\beta(L_+,I)\le p$ we get 
\[\beta(L_+,R)\le (q-1)s+\max\{p,\beta(L_+/R_+L_+,\field[I_{\le s}])-s\}\]
so our statement  follows from the inequality $\beta(L_+/R_+L_+,\field[I_{\le s}])\le p+s$. 

To prove the latter observe that if $h(m)>s$ for a monomial $m$, then $| \Phi(m)^O |>s$ for some $G/A$-orbit $O$ in $\hat A$. Therefore  
\begin{equation}\label{eq:N+M}L_{\ge p+s}=N+\sum_{i=0}^{p-1/q}M^{(i)}\end{equation} 
where $N$ is spanned by monomials having  a degree $p+s$ divisor $m$ with $h(m)\le s$, 
$M^{(0)}$ is spanned by monomials involving an $A$-invariant variable, and for $i=1,\dots,\frac{p-1}q$, $M^{(i)}$ is spanned by monomials having a divisor $m$ with 
$\deg(m)\ge p+s$ and $|\Phi(m)^{O_i}|>s$; here $O_1,\dots,O_{p-1/q}$ are the $q$-element $G$-orbits in $\hat A$. 
 
By  Lemma \ref{Zp_corner} the weight sequence $\Phi(m)$ of a monomial $m\in N$ contains a non-empty zero-sum sequence
of length at most $h(m)\le s$, hence $m\in \field[I_{\le s}]_+L_+$. Applying Proposition~\ref{prop:gapless} with $\mathcal{V}$ the variables with weight in $O_i$  
for a fixed $i\in\{1,\dots,\frac{p-1}q\}$, we get $M^{(i)}\subseteq L_+R_++\mathcal{G}_{s+1}(\mathcal{V})+M^{(0)}$,  
and by Proposition~\ref{binom} we have $\mathcal{G}_{s+1}(\mathcal{V})\subseteq \mathcal{B}$. Clearly $M^{(0)}\subseteq\mathcal{B}$. It follows by \eqref{eq:N+M} that 
$L_{\ge p+s}\subseteq R_+L_++\mathcal{B}+L_+\field[I_{\le s}]_+$, and since bricks have degree at most $q\le s$, the inequality 
$\beta(L_+/R_+L_+,\field[I_{\le s}])\le p+s$ is proven. 
\end{proof}

\begin{remark} \label{rem:skatulya}
The above results are getting close to the lower bound mentioned at the beginning of Chapter~\ref{ch:semidir}
only for small values of $q$: 
we have $ p+2 \le \beta(Z_p \rtimes Z_3) \le p+6$  by Corollary~\ref{cor:tuske}  
and $p+3 \le  \beta(Z_p \rtimes Z_4) \le p+6 $ by Corollary~\ref{cor:skatulya} (for the lower bounds see \cite{CzD:2}). 
In characteristic zero, the inequality $\beta(Z_p\rtimes Z_3)\le p+6$ was proved in \cite{pawale}. 
\end{remark}

\begin{proposition} \label{nagy}  We have 
 \[\mathcal{G}_d \subseteq (I_+)_{\le q}L
 \qquad \text{ if }\quad
 d \ge \min \{ p, \tfrac{1}{2}(p+q(q-2)) \}.\] 
\end{proposition}

\begin{proof}
Suppose that $m$ is a gapless monomial having no non-trivial $A$-invariant divisor of degree at most $q$ 
(hence $m$ is not divisible by a brick). 
In particular $m$ has no variables of weight $0$. 
Let $m=m_1...m_{p-1/q}$ be the factorization where $\Phi(m_i)=\Phi(m)^{O_i}$,  
and let $S_i$ denote the support of the weight sequence $\Phi(m_i)$. 
By our assumption $0 \not\in S:=\bigcup_j S_j$ and $|S_i| \le q-1$ for every $i$. 

For each factor $m_i$ we have $h(m_i) \le |S_i| \le q-1$, so if $\deg(m) \ge p$ then
$m$ contains an $A$-invariant divisor of degree at most $h(m) \le q-1$ by Lemma~\ref{Zp_corner},
which is a contradiction, hence $\deg(m) \le p-1$. 
On the other hand, 
as each factor $m_i$ is gapless, $\deg(m_i)\leq \binom{|S_i|+1}{2}\leq 
\frac{|S_i|q}2$, and consequently 
\begin{align} \label{sq/2} 
\deg(m) \le \frac{|S|q}{2}. 
\end{align}

We  claim that $|S| \le q +\frac{p-1}q -2$. Write $q^\wedge T:=\{t_1+\dots +t_q\mid t_i\neq t_j\in T\}$
for any subset $T \subseteq \hat A$.  The Dias da Silva - Hamidoune theorem (see \cite{dasilva}) states that
$ |q^{ \wedge} T| \ge \min \{p, q|T|- q^2 + 1 \}$.
Now if our claim were false then we would get from  this theorem that 
\[ |q^\wedge (S \dot\cup \{ 0\}) | \ge \min\{p,q(|S|+1) - q^2 + 1 \}=p\]
implying that $m$ contains an $A$-invariant divisor of degree $q$ or $q-1$, again a contradiction.
By plugging in this upper bound on $|S|$ in \eqref{sq/2} and since $q$ is odd we get 
$ \deg(m) \le\lfloor \frac{q^2-2q+p-1}{2}\rfloor 
=\tfrac{1}{2}(p+q(q-2)) -1$. 
\end{proof}

\begin{proposition} \label{kicsi}
Suppose  $c, e$ are positive integers  such that $c\leq q$ and  
$ \binom{c}{2} < p \le \binom{c+1}{2} - \binom{e+1}{2} $ 
(in particular, this forces that $p<\binom{q+1}2$). 
Then
\[ \mathcal{G}_ d \subseteq (I_+)_{\le c-e} L \qquad \text{ if } \qquad d \ge p+ \binom{e}{2}.\]  
\end{proposition}

\begin{proof}
Suppose that $m$ is a gapless monomial having  no non-trivial $A$-invariant divisor of degree at most $c-e$. 
Take  the row-decomposition  $\Phi(m)=S_1\cdots S_h$ and set $E:=S_1\cdots S_{c-e}$, $F:=S_{c-e+1}\cdots S_h$.  
We have $|E| \le p-1$, for otherwise by Lemma~\ref{Zp_corner}
we would get an $A$-invariant divisor of degree at most $c-e$. 
It follows that $|S_{c-e}| \le e$, for otherwise the fact that $m$ is gapless and $c\le q$ would 
lead to the contradiction 
\[ |E| \ge (e+1) + (e+2) + ... + (e+(c-e)) = \textstyle\binom{c+1}{2} - \binom{e+1}{2} \ge p.\] 
As a result $|S_{c-e+1}| \le e-1$, hence $|F| \le \binom{e}{2}$ since $m$ is gapless. 
But then $\deg(m) =|E|+|F| \le p-1 + \binom{e}{2}$, and this proves our claim. 
\end{proof}

To illustrate the use of  Proposition~\ref{kicsi} consider the case when $p=11$ and $q=5$. 
We then have $c=5$ and $e=2$, hence any gapless monomial of degree at least $12$ 
contains an $A$-invariant of degree at most $3$. On the other hand 
 $I_{\ge 22}\subseteq I_+R_++(\mathcal{G}_{12}\cap I_{\ge 22})\subseteq I_+R_++(I_+)_{\le 3}I_{\ge 19}$
 by Proposition~\ref{prop:gapless}, 
hence $I_{\ge 28} \subseteq I_+^3I_{\ge 19} +I_+R_+$. 
Furthermore  $I_{\ge 19}\subseteq I_+R_++(\mathcal{G}_9\cap I_{\ge 19})$ by Proposition~\ref{prop:gapless}. 
 A monomial $m\in \mathcal{G}_9\cap I_{\ge 19}$ 
has  a gapless divisor $u$ of degree at least $9$. 
It is easily seen that $h(u) \le 3$, hence $u$ can be completed to a monomial $v \mid m$ of degree $11$ and height $h(v) \le 5$, 
which will contain an $A$-invariant divisor of degree at most $5$ by Lemma~\ref{Zp_corner}. We get that  $I_{\ge 19} \subseteq (I_+)_{\le 5}I_{\ge 14} + I_+R_+$. 
Finally $I_{\ge 14} \subseteq I_+^2$ and putting all these together yields 
 $I_{\ge 28} \subseteq I_+^6 +I_+R_+\subseteq I_+R_+$ by Proposition~\ref{prop:knopeta}. As a result
\begin{align} \label{betaZ11Z5}
\beta(Z_{11} \rtimes Z_5) \le 27.
\end{align}

\begin{proposition}\label{ZpZq_estimates}
For any odd primes $p,q$ such that $q \mid p-1$ we have the following estimates:
\[ \beta(L_+,R)  \le
\begin{cases}
\frac{3}{2} (p + (q-2)q )-2  & \text{ if } p>q(q-2); \\
2p + (q-2)q -2 & \text{ if } p < q(q-2); \\
2p + (q-2)(c-1) -2 & \text{ if }  c(c-1) < 2p < c(c+1), c \le q.
\end{cases}
\]
\end{proposition}
\begin{proof}
Let $d$ be a positive integer such that $\mathcal{G}_d \subseteq (I_+)_{\le q}I$.
Given that $\mathcal{B} \subseteq (I_+)_{\le q}I$,
we get  $\beta(L_+/R_+L_+,\field[I_{\le q}]) \le p+d-2$ by Proposition~\ref{prop:gapless}.
Using Lemma~\ref{induk} it follows $\beta(L_+,R) \le (q-2)q+ p+d-2$.
Our first two estimates follow by substituting the value of $d$ given in Proposition~\ref{nagy}.
The last one follows similarly by deducing from Proposition~\ref{kicsi} that 
$\beta(L_+/R_+L_+,\field[I_{\le c-1}]) \le 2p -2$,
and then applying Lemma~\ref{induk}.
\end{proof}

\begin{theorem} \label{thm:zpzq}
For the non-abelian semidirect product $Z_p\rtimes Z_q$, where $p,q$ are odd primes we have 
$\beta(Z_p \rtimes Z_q) < \tfrac{pq}{2}$. 
\end{theorem}

\begin{proof} 
Recall that $\beta(G,W)\le\beta(L_+,R)$ by Proposition~\ref{trans}. 
Hence by Corollary~\ref{cor:tuske} we have $\beta(Z_p \rtimes Z_3) \le p+6$,
hence $\beta(G) < |G|/2$ for $p>7$.
The case $p=7$ will be treated below, with the result $\beta(Z_7 \rtimes Z_3) =9$ in Theorem~\ref{thm:z7z3}.
For the rest we may assume that $q \ge 5$.
Suppose indirectly that $pq \le 2\beta(Z_p \rtimes Z_q)$. Then  by the first estimate 
in Proposition~\ref{ZpZq_estimates}
\[p(q-3) \le 3q(q-2) -4.\]
Suppose first  that $4q+1 \le p$. In this case   $q^2 -5q +1 \le 0$,
whence $q<5$, a contradiction.
It remains that $p=2q +1$.  
Since by \eqref{betaZ11Z5} our statement is true for $q=5, p=11$,
it remains that $q \ge 11$ (as $2q+1$ is not prime for $q=7$). 
Then $2p < q(q+1)$, so we can apply the third estimate in Proposition~\ref{ZpZq_estimates}.
By the indirect assumption and the fact that $c(c-1) < 2p$ we get that
\[ \frac{pq}{2} < 2p + (q-2)\frac{2p}{c}.  \]
Here $c \ge 7$ as $p \ge 23$, but then by this inequality $q \le 6$, a contradiction. 
\end{proof}


\subsection{The group $Z_7 \rtimes Z_3$}\label{sec:z7z3} 

In this section we will deal with the group $G = Z_7 \rtimes Z_3$, 
and suppose that $\mathrm{char}(\field)\neq 3,7$. The character group $\hat A$ of the abelian normal subgroup $A=Z_7$ of $G$ 
will be identified with the additive group of residue classes modulo $7$, 
so the generator $b$ of $G/A=Z_3$ acts on $\hat A$ by multiplication with $2 \in(\mathbb{Z}/7\mathbb{Z})^{\times}$. 
Then we have three $G/A$-orbits in $\hat A$, namely  $A_0:=\{ 0\}$, $A_{+} := \{1,2,4 \}$, $A_{-} := \{ 3,5,6\}$. 
Accordingly $G$ has two non-isomorphic irreducible representations of dimension $3$, denoted by $V_{+}$ and $V_{-}$.
Let $W$ be an arbitrary representation of $G$; it has a decomposition 
\begin{align}
W=  V_+^{\oplus n_+} \oplus V_-^{\oplus n_-} \oplus V_0 \end{align}
where $V_0$ is  a representation of $Z_3$ lifted to $G$. 
Any monomial $m\in \field[W]$ has a canonic factorization 
$m=m_+m_- m_0$ given by the canonic isomorphism $\field[W] \cong  \field[V_+^{\oplus n_+}] \otimes \field[V_-^{\oplus n_-}] \otimes \field[V_0]$;
the degrees of these factors will be denoted by $d_+(m), d_-(m), d_0(m)$.  
Finally we set $I = \field[W]^A$, $R = \field[W]^G$ and let $\tau=\tau^G_A: I \to R$ denote the transfer map.

\begin{proposition} \label{prop:lapos}
Let $m\in \field[W]$ be a $Z_7$-invariant monomial with   $\deg(m) \ge 7$, $d_0(m)=0$ and $d_+(m),d_-(m) \ge 1$. 
Then $m\in I_2I_++I_+R_+$. 
\end{proposition}

\begin{proof} 
Denote by $S$ the support of the weight sequence $\Phi(m)$ 
and by $\nu_w$ the multiplicity of $w \in \hat{A}$ in $ \Phi(m)$.
Observe that $|S|\ge 2$ since $d_+(m),d_-(m)$ are both positive.
This also implies that $m \in I_+^2$, since any irreducible zero-sum sequence of length at least $7$ 
is similar to $(1^7)$. 
We have the following cases:

(i) if $|S| \ge 4$ then $S \cap - S \neq \emptyset$ hence already $m \in I_2I_+$. 

(ii) if $|S|=3$ then up to similarity,
we may  suppose  that $S \cap A_+ = \{1\}$ and $S \cap A_{-} = \{ 3,5\}$. 
If  a  factorization  $m=uv$ exists where $u,v$ is $Z_7$-invariant and
$1 \in \Phi(u)$, $(35) \subseteq \Phi(v)$ then obviously $m - u \tau(v) \in I_2I_+$.
This certainly happens if $\Phi(m)$ contains $(1^7)$ or one of the irreducible zero-sum sequences with support $\{ 3,5\}$, namely
$
(3 5^5), 
(3^2 5^3)$, or 
$(3^3 5)$.
Otherwise it remains that $\nu_1 \le 6$, $\nu_3 \le 2$ and $\nu_5 \le 4$.  
Now, if $\Phi(u)=(135^2)$ then necessarily either $1\in \Phi(v)$ or $(35) \subseteq \Phi(v)$, 
and in both cases $m - u \tau(v) \in I_2I_+$.
It remains that $\nu_5= 1$, and therefore $\Phi(m)$ equals $(1^3 3^2 5)$ or $(1^6 3 5)$. 
The first case is excluded since $\deg(m) \ge 7$. 
In the second  take $\Phi(u) = (1^4 3)$ , $\Phi(v) = (1^2 5)$
and observe that $\Phi(uv^{b^2}) $ falls under case (i), 
while $\Phi(uv^b) = (1^4 2^2 3^2)$ is similar to the sequence $(1^2 3^2 5^4)$
which was already dealt with.

(iii) if $|S| =2$ then again $m = uv$ for some $u,v \in I_+$. 
Denote by $U$ and $V$ the support of $\Phi(u)$ and $\Phi(v)$, respectively. 
If $|U| \ge 2$ or $|V| \ge 2$ then after replacing $m$
by $m - u \tau(v)$ we get back to case (ii) or (i). 
Otherwise $\Phi(m) = (a^{7i} b^{7j})$ for some $a\in A_+, b \in A_-$ and $i,j \ge 1$; 
but then an integer $1 \le n \le 6$ exists such that $(ab^n)(a^{7i-1}b^{7j-n})$ is a $Z_7$-invariant factorization,
and we are done as before. 
\end{proof}

\begin{corollary}\label{cor:lapos}  
If $m\in \field[W]$ is a $Z_7$-invariant monomial such that
$\deg(m) \ge 10$ and
$d_0(m) \ge 2$ or $ \min\{d_+(m),d_-(m)\} \ge 3 - d_0(m)$
then $m\in I_+R_+$. 
\end{corollary}

\begin{proof} 
By  Corollary~\ref{cor:G:H+1} it is enough to prove that $m \in I_+^4$.
This is immediate if $d_0(m) \ge 2$. 
If $ d_0(m) =1$ then  applying Proposition~\ref{prop:lapos} two times 
shows that $m \in I_1I_2^2I_+$. 
Finally, if $d_0(m) =0$ then again after two applications of Proposition~\ref{prop:lapos}
we may suppose that $m=uv$ where  $\deg(v) \ge 6$, $d_+(v),d_-(v) \ge 1$ and $u \in I_2^2$.
It is easily checked 
that any irreducible zero-sum sequence over $Z_7$ of length at least $6$
is similar to $(1^7)$ or $(1^5 2)$, 
none of which equals $\Phi(m)$ (for then $d_-(m) = d_-(u) = 2$, a contradiction).
Therefore $v \in I_+^2$ follows and again $m \in I_+^4$.
\end{proof} 

\begin{lemma} \label{lemma:trukk}
Let $G= A \rtimes \langle g \rangle $ where $ \langle g\rangle \cong Z_3 $ and $A$ is an arbitrary abelian group.
If $3 \in \field^{\times}$ then for any  $u,v,w \in I_+$ 
\[ uvw -  u v^g w^{g^2} \in I_+ (R_+)_{\le deg(vw)}.  \] 
\end{lemma}

\begin{proof}
The following identity can be checked by mechanic calculation:
 \begin{align*}
 3\left(uvw - uv^gw^{g^2} \right)
 &= uv\tau(w)+uw\tau(v)+u\tau(vw)\\
& -u\tau(vw^g)-uw^{g^2}\tau(v)-uv^g\tau(w).  \qedhere
 \end{align*}
\end{proof}

\begin{proposition} \label{prop:hegyes} 
Let $m\in \field[W]$ be a $Z_7$-invariant monomial with $m_+$ factorized as  $m_+=m_1...m_n$ (where $n:=n_+$)
through the isomorphism $\field[V_+^{\oplus n}] \cong \field[V_+]^{\otimes n}$. 
If $\deg(m)\ge 10$, $d_0(m) \le 1$ and $\max_{i=1}^n \deg(m_i) \ge 3$ then $m\in I_+R_+$.
\end{proposition}

\begin{proof}
We shall denote by $x,y,z$ the variables of weight $1,2,4$  belonging to that copy of $V_+$ for which $\deg(m_i)$ is maximal, 
while $X,Y,Z$ will stand for the variables of the same weights  which belong to any other copy of $V_+$. 

Since $d_0(m) \le 1$ by assumption, 
using Proposition~\ref{prop:gapless}  with $\mathcal{V}:=\{x,y,z\}$ we may assume that $m_{\mathcal{V}}$ has a gapless divisor $t$ of degree at least $3$. 
Let $S\subseteq \hat A$ be the support of the weight sequence $\Phi(t)$; clearly  $|S| \ge 2$. 
If $|S| = 3$ then $m_{\mathcal{V}}$ is divisible by the $G$-invariant  $xyz$, and we are done.
It remains that $|S| = 2$ hence by symmetry we may suppose that $m_{\mathcal{V}}$ is divisible by $t=x^2y$. 

If $d_0(m)=1$ then $m$ contains an $A$-invariant variable $w$ and by Lemma~\ref{lemma:easy} $|\Sigma(\Phi(m/tw))| =7 $. 
This gives an $A$-invariant factorization $m/w = uv$ such that $xy \mid u$ and $x \mid v$. 
By Lemma \ref{lemma:trukk} we get that 
$m \equiv uv^bw^{b^2} \mod I_+R_+$,
where $uv^bw^{b^2}$ contains $xyz$ for a suitable choice of  $b \in  \{ g, g^2\}$, so we are done. 

It remains that $d_0(m) =0$. By a similar argument as in the proof of  Proposition~\ref{prop:gapless}, we may assume 
that $m_+$ has a gapless divisor of degree $4$,
while $m_{\mathcal{V}}$ still contains a gapless divisor of degree $3$. 
Therefore we may suppose that $ m_+$ contains $u:=xyZ$ while $m_{\mathcal{V}}$ still contains $x^2y$. 
Now if $m/u \in I_+^2$ then we get an $A$-invariant factorization $m=uvw$  such that $xy \mid u$ and $x \mid v$,
so we are done again by using Lemma \ref{lemma:trukk}. 
Finally, if $m/u$ is irreducible then necessarily $\Phi(m/u) = (1^7)$, so that $m = x^2 y X^6 Z$.
Here we can employ the following relations: 
\begin{align*}
x^2 y X^6 Z & =  
xy X^4 \; \tau(x X^2Z )  -
xy z 	X^4 	 Z^2Y   -
xy^2		 X^5 	 Y^2   \\
xy^2		 X^5 	 Y^2 & =
xy Y^2  \; \tau(yX^5) -
xy z Y^7  -
x^2 y Y^2  Z^5.
\end{align*}
This proves that $m \equiv x^2 y Y^2  Z^5 \mod{I_+R_+}$, and as $xY^2Z^4 \in I_+^2$, the latter monomial 
already belongs to $I_+R_+$ by the first part of this paragraph. 
\end{proof}

\begin{corollary} \label{z7z3reg}
If $W$ is the regular representation $V_{\mathrm{reg}}$ 
of $Z_7 \rtimes Z_3$ then we have 
 $\beta(I_+, R) \le 9$.
\end{corollary}

\begin{proof}
Here we have   $n_+ = n_- = 3$. 
Let $m \in I_+$ be a monomial with $\deg(m) \ge 10$. 
If  Corollary~\ref{cor:lapos} can be applied then $m \in I_+R_+$ already holds.
Otherwise $d_0(m) \le 1$ and say $d_-(m) \le 2-d_0(m)$, whence $d_+(m) \ge 8$. 
Then one of the monomials in the factorization $m_+=m_1m_2m_3$,  say $m_{1}$ has degree at least $3$,
and we are done by Proposition~\ref{prop:hegyes}. 
\end{proof}

It was observed by Schmid that $\beta(G) = \beta(G, V_{\mathrm{reg} })$ for any finite group $G$ if  $\mathrm{char} (\field) = 0$.
This is based on  Weyl's theorem on polarization (see \cite{weyl}). 
If $\mathrm{char}(\field) > 0$,  then Weyl's theorem on polarizations fails even in the non-modular case; instead of that, if $\mathrm{char}(\field)$ does not divide $|G|$
then by a result of  Grosshans in \cite{grosshans} 
for any $G$-module $W$ containing $V_{\mathrm{reg}}$ as a submodule,
 the ring $\field[W]^G$ is the $p$-root closure of its subalgebra generated by the  polarization of  
$\field[V_{\mathrm{reg}}]^G$. 

Corollary~\ref{z7z3reg} is an improvement of Pawale's result  who proved in \cite{pawale} in characteristic $0$ that $\beta(G,W) = 9$ for $n_+,n_- = 2$, 
and from this he concluded $\beta(G)= 9$ using a  version of Weyl's Theorem on polarization. 
For positive characteristic we will use the following result: 

\begin{proposition}[Knop, Theorem 6.1 in \cite{knop}]\label{prop:knop}
Let $U$ and $V$ be finite dimensional $G$-modules. If $ n_0 \ge \max \{ \dim (V), \frac{\beta(G)}{\mathrm{char}(\field) -1} \}$ and 
$S$ is a generating set of $\field[U \oplus V^{\oplus n_0}]^G$ then  $\field[U\oplus V^{\oplus n}]^G$ for any $n \ge n_0$
is generated by the polarization (with respect to the type-$V$ variables) of $S$. 
\end{proposition}

\begin{proposition} \label{prop:Z7Z3_ch>2}
If $\mathrm{char}(\field) \neq 2,3,7$ then $\beta(G) \le 9$.
\end{proposition}

\begin{proof}
We already know that $\beta(G)\leq 13$ from Corollary~\ref{cor:tuske}. 
Therefore it is sufficient to show that $R_d \subseteq R_+^2$ whenever $10 \le d \le 13$.
Suppose first that $\mathrm{char}(\field)>7$. Then $\max \{ \dim(V_+) , \dim(V_-), \frac{\beta(G)}{\mathrm{char}(\field) -1} \} = 3 $
hence by Proposition~\ref{prop:knop} a generating  set of $\field[W]^G$ can be obtained by polarizations
from a generating set of $\field[V_{\mathrm{reg}}]^G$, so $\beta(G) \le \beta(G,V_{\mathrm{reg}}) \le 9$ by Corollary~\ref{z7z3reg}.

Finally let $\mathrm{char}(\field)=5$, so that $\max \{ \dim(V_+), \dim(V_-), \frac{\beta(G)}{\mathrm{char}(\field) -1}\} \le 4$. 
By Proposition~\ref{prop:knop} here we can obtain the generators of $R$ by polarizing
the generators of $S := \field[V_+^4 \oplus V_-^4 \oplus V_0]^G$. 
$S$ is spanned by elements $f$ that are multihomogeneous in the sense that
for all monomials $m$ occurring in $f$ the triple $(d_+(m),d_-(m),d_0(m))$ is the same; denote it by  $(d_+(f),d_-(f),d_0(f))$.
We know from formula (6.3) and Theorem 5.1 in \cite{knop}  that 
$f$ is contained in the polarization of $\field[V_{\mathrm{reg}}]$ (taken with respect to $V_+^{\oplus 3}$ and then to  $V_-^{\oplus 3}$ separately), 
if $d_+(f), d_-(f) \le 3(\mathrm{char}(\field) - 1) = 12$. 
So for the rest we may suppose that say $d_+(f) \ge 13$. 
Then let $f_+ = f_1f_2f_3f_4$ be the factorization given by the isomorphism $\field[V_+^{\oplus 4}] \cong \field[V_+]^{\otimes 4}$, 
and observe that $\deg(f_i) \ge 4$ for some $i\le 4$,
whence $f \in I_+R_+$  by Proposition~\ref{prop:hegyes}. 
 \end{proof}

\subsection{The case of characteristic $2$}

The polarization arguments at the end of the previous section does not cover the case  $\mathrm{char}(\field) =2$. 
Here we  need a closer look at the interplay between our extended Goebel algorithm 
and the elementary polarization operators 
\[ \Delta_{i,j} := x_j \frac{\partial}{\partial x_i}+ y_j \frac{\partial}{\partial y_i} + z_j \frac{\partial}{\partial z_i}\]
where as before $\field[V_+^{\oplus n}] = \otimes_{i=1}^n\field[x_i,y_i,z_i]$ and the variables $x_i,y_i,z_i$ have weight $1,2,4$, respectively. 
The operators $\Delta_{i,j}$ are $G$-equivariant, hence map
$G$-invariants to $G$-invariants.
Moreover, by the Leibniz rule it also holds that:
\begin{align} \label{pol}   \Delta_{i,j}(I_+R_+) \subseteq I_+R_+. \end{align}

\begin{proposition} \label{prop:Z7Z3_ch=2}
If $\mathrm{char}(\field) = 2$  then
$\beta(I_+,R) \le 9 $. 
\end{proposition}

\begin{proof} 
Let $m \in I$ be a monomial with $\deg(m) \ge 10$. It is sufficient to show that $m\in I_+R_+$. 
We may suppose by symmetry  that $d_+(m) \ge d_-(m)$.
It suffices to deal with the cases not covered by Corollary~\ref{cor:lapos} 
so we may suppose that $d_0(m) \le 1$, $d_-(m) \le 2 - d_0(m)$, whence 
$d_+(m) \ge 8$. 
By  Proposition~\ref{prop:gapless} we can assume that $m_+$ contains a gapless monomial of degree $3$. 
We  have  several   cases:

(i) Let $m_+=m_1...m_n$ where each monomial $m_i$ belongs to a different copy of $V_+$. 
If $\deg(m_i) \ge 3$ for some $i \ge 1$ then $ m\in I_+R_+$ by Proposition~\ref{prop:hegyes}. 
So for the rest we may suppose that $\deg(m_i) \le 2$ for every $i=1,...,n$.

(ii) If $m_+$ contains the square of a variable, say $x_1^2$ then
a variable of weight $2$ or $4$ must also divide $m$, say $m =x_1^2 y_2 u$, 
because we assumed that $m_+$ contains a gapless divisor of degree $3$. 
Here we have
\begin{align*}
\Delta_{1,2} x_1^2 y_1 u = 2 x_1y_1x_2u + x_1^2 y_2 u = m
\end{align*}
as $\mathrm{char}(\field)=2$. In view of case (i) and \eqref{pol} this shows  that  $m \in I_+R_+$. 

(iii) If $m_+$ is square-free, but still $\deg(m_i) =2$ for some $i$, say $x_1y_1 \mid m$, then 
our goal will be to find three monomials $u,v,w \in I_+$ such that $m=uvw$ and $x_1 \mid u$, $y_1 \mid v$.
For then $m \equiv uv^bw^{b^2} \mod I_+R_+$ by Lemma~\ref{lemma:trukk} 
where $b$ can be chosen so that $uv^bw^{b^2}$ contains $x_1^2$, and then $m$ will  fall under case (ii). 
Here are some conditions under which this goal can be achieved:
\begin{enumerate}
\item[(a)] If $d_0(m) =1$ then let $w$ be the  $Z_7$-invariant variable  in $m$;
given that $|\Sigma(\Phi(m/wx_1y_1))|=7$ by Lemma~\ref{lemma:easy},  suitable factors $u,v$ must exist. 

\item[(b)] It remains that $d_0(m) = 0$. Again by  Proposition~\ref{prop:gapless}  (with $\mathcal{V}$   the set of variables in $\field[ V_+^n]$)
we assure that $m_+$ contains a gapless monomial of degree $4$, hence also a $Z_7$-invariant  $u := x_1y_1Z$. 
Suppose now that $m/u=vw$ for some $v,w \in I_+$. 
Up to equivalence modulo $I_+R_+$ we may also suppose  that 
one of these two monomials, say $v$ contains a variable $X$ (or $Y$). 
After swapping $x_1$ and $X$ (or $y_1$ and $Y$) in $u$ and $v$ we are done.

\item[(c)] If  $d_-(m)>0$, then $m/u$ has a variable $t$   such that 
some $f\in \{x_1t,y_1t, Zt \}$ belongs to $I$; as $\deg(m/f)\ge 8$, the desired factorization of $m$ is given by Lemma~\ref{lemma:easy}. 

\item[(d)] It remains that $d_0(m)=d_-(m)=0$ and $\Phi(m/u)$ is an irreducible zero-sum sequence. 
Since $\deg(m/u) \ge 7$ it follows that $\Phi(m/u)$ equals $(2^7)$, $ (1^7)$ or $(4^7)$. 
In the first  case we use the relation:
\[m=x_1y_1Z Y^7 = \tau(x_1 Y^3) y_1ZY^4 - y_1^2Y^4Z^4 - z_1y_1 X^3Y^4Z  \]
where the two monomials on the right hand side fall under case (ii) or (iii/b). 
The case $\Phi(m/u) = (1^7)$  is similar. Finally, if $\Phi(m/u) = (4^7)$ then we replace $m$ with $m - u\tau(m/u)$ 
to reduce to the other two cases. 
\end{enumerate}

(iv) If $m$ is multilinear: here we can again assume that $(124) \subseteq \Phi(m)$. 
If $d_0(m) =0$ then this is achieved using Proposition~\ref{prop:gapless}.
Otherwise, if there is a $Z_7$-invariant variable $w$  in $m$
then we may still suppose by Proposition~\ref{prop:gapless}  that e.g. $x_1y_2x_3\mid m$
and the same argument as above at (iii/a) gives a factorization $m/w = uv$ such that 
$x_1y_2 \mid u$ and $x_3 \mid v$, so our goal is achieved by Lemma~\ref{lemma:trukk}. 
Now we may suppose that $m=x_1y_2z_3u$, say. We have:
\begin{align*}
\Delta_{1,2} z_1 x_1 y_3u   + \Delta_{3,1} z_2 x_3 y_3 u &= 
(z_1 x_2 y_3  
+ z_2 x_3 y_1 )u 
 = m + \tau(x_1y_2z_3) u 
\end{align*}
The monomials $z_1 x_1 y_3 u$ and $z_2 x_3 y_3 u$ fall under case (iii),
so $m \in I_+R_+$. 
\end{proof}

Comparing Proposition~\ref{prop:Z7Z3_ch>2} and Proposition~\ref{prop:Z7Z3_ch=2} with the lower bound
mentioned at the beginning of Chapter~\ref{ch:semidir}, we have proved:

\begin{theorem} \label{thm:z7z3}
If $\mathrm{char}(\field) \neq 3,7$ then $\beta(Z_7 \rtimes Z_3)=9$.   
\end{theorem}

In a subsequent paper \cite{cziszter}  the first author proved that $\beta(Z_p\rtimes Z_3)=p+2$ for any prime $p$ congruent to $1$ modulo $3$. 

\section{Some further particular cases}


\subsection{The group $Z_5 \rtimes Z_4$, where $Z_4$ acts faithfully}

The following is proved (without explicitly being stated) by Schmid \cite{schmid} for ${\mathrm{char}}(\field)=0$ and by Sezer \cite{sezer} in non-modular positive characteristic:

\begin{proposition}\label{prop:betaD2n} 
Suppose that $2n\in\field^\times$. For any module $V$ of the dihedral group $D_{2n}=Z_n\rtimes_{-1}Z_2$ we have 
\[\beta(D_{2n}, V) \le \beta(\field[V]^{Z_n}_+,\field[V]^{D_{2n}})\le n+1.\]
\end{proposition}

Let $G:=Z_5 \rtimes Z_4$ where $Z_4=\langle b\rangle $ 
and conjugation by $b$ is an order $4$ automorphism of the normal subgroup $A=Z_5$. Take a $G$-module $V$ and 
set $L:=\field[V]$, $R:=\field[V]^G$, $I:=\field[V]^A$, $S:=\field[V]^H$, where 
$H\cong D_{10}$ is the subgroup of $G$ generated by $A$ and $b^2$. 

\begin{proposition} \label{prop:z5z4} 
If $\mathrm{char}(\field)\neq 2,5$ then  $\beta(I_+,R)=8$. 
\end{proposition}

\begin{proof} The lower bound $\beta(I_+,R)\ge 8$ follows from a result in \cite{CzD:2}. 
 By Corollary~\ref{cor:skatulya} we have $\beta(I_+,R)\le 5+6=11$. Therefore it is sufficient to show that 
if $m$ is an $A$-invariant monomial with $9\le\deg(m)\le 11$, then $m\in I_+R_+$. 
Suppose there are three variables $e,f,h$ such that $m=efhr$ and both $ef$ and $eh$ are $A$-invariant. 
The relation \begin{align} \label{masodik} 2m =  
\tau^H_A(ef) hr + 
\tau^H_A(eh) fr - 
\tau^H_A(fh^{b^2}) e^{b^2} r. 
\end{align}
implies that 
$m\in S_2I_{\ge 7}$, and since $\beta(I_+,S)\le 6$ by Proposition~\ref{prop:betaD2n} 
we get $m\in I_+S_+^2\subseteq I_+R_+$ (the latter inclusion follows by Proposition~\ref{prop:knopeta}). 
If $m$ contains two $A$-invariant variables then $m\in I_1^2I_{\ge 7}\subseteq I_{\ge 7}S_+$ by 
Proposition~\ref{prop:knopeta}. As above, $I_{\ge 7}\subseteq I_+S_+$, so $m\in I_+S_+^2\subseteq I_+R_+$. 
From now on suppose that none of the above two cases hold for $m$. 
Then $m=m_0m_+$, where $m_0=1$ or $m_0$ is an $A$-invariant variable, $m_+$ involves no $A$-invariant variables, and 
$8\le \deg(m_+)\le 11$. This forces that the support of $\Phi(m_+)$ has at most two elements (not opposite to each other). 
 The action of $G/A$ preserves $I_+R_+$, therefore it is sufficient to deal with the case when 
 $\Phi(m_+)=(1^k,2^l)$ or $\Phi(m_+)=(1^k,3^l)$ where $k\ge l$. 
 If $l\ge 2$ then $m_+=ef$ where $e,f$ are $A$-invariant and $\supp(\Phi(e))=\supp(\Phi(f))=\supp(\Phi(m_+))$; 
 now each monomial of $m-e\tau^G_A(f)$ belongs to $I_+R_+$ by the cases considered already. Finally, 
 if $l\le 1$, then $m_+=ef$ where $\Phi(f)=1^5$; again all monomials of $\tau^G_A(f)$ belong to $I_+R_+$ by the prior cases. 
\end{proof}

\subsection{The alternating group $A_4$}\label{sec:a4} 

Throughout this chapter let $G:=A_4$, the alternating group of degree four. 
The double transpositions and the identity constitute a normal subgroup $A\cong Z_2\times Z_2$ in $G$, and 
$G=A\rtimes Z_3$ where $Z_3 = \{ 1, g,g^2\}$. Denote by $a,b,c$ the involutions in $\hat A$, conjugation by $g$ permutes them cyclically. 
Remark for future reference that the only irreducible zero-sum sequences over $\hat A$ are: $(0)$, $(a,a)$, $(b,b)$, $(c,c)$, $(a,b,c)$.
Hence the factorization of any zero-sum sequence  over $Z_2 \times Z_2$ into maximally many irreducible ones is of  the form
\begin{equation}\label{eq:klein factorization}
 (0)^q(a,a)^r(b,b)^s(c,c)^t(a,b,c)^e \qquad \text{ where } e =0 \text{ or } 1. 
\end{equation} 
In particular the multiplicities of $a,b$ and $c$ must have the same parity.

Let $\field$ be a field with characteristic  different from $2$ or $3$. Apart from the one-dimensional representations of $G$ factoring through the natural surjection $G\to Z_3$, 
there is a single irreducible $G$-module $V$, hence an arbitrary finite dimensional $G$-module $W$ shall decompose as 
\[ W= U \oplus V^{\oplus n} \]
where $U=W^A$ consists of one-dimensional $G$-modules. 
 $V$ is the $3$-di\-men\-sio\-nal summand in the natural $4$-dimensional permutation representation of $G$.  Let $x,y,z$ denote the corresponding basis in $V^*$ 
and following our conventions introduced in Section \ref{sec:goebel} 
let $\field[V^{\oplus n}] = \otimes_{i=1}^n \field[x_i,y_i,z_i]$,
so that $x_i,y_i,z_i$ are $A$-eigenvectors of weight $a,b,c$ which are permuted cyclically by $g$. 
We write $I:=\field[W]^A$, $R:=\field[W]^G$ and $\tau:=\tau^G_A:I\to R$.

\begin{proposition}\label{prop:etaa4} If  $n = 3$
then $R_{\ge 7} \subseteq (R_+)_{\le 4} R_+$. 
\end{proposition} 

\begin{proof}  
It is sufficient to show that $ I_{\ge 7} \subseteq (R_+)_{\le 4}I+ (I_+)_{\le 4}R$. 
Take a monomial $m \in I_{\ge 7}$  with $\deg(m_+) \ge 7$.
We claim that   $m \in I_+(R_+)_{\leq 4}$ in this case. 
Consider the factorization $m_+=m_1m_2m_3$  given by the map $\field[V^{\oplus 3}] \cong \field[V]^{\otimes 3}$;
by symmetry we may assume that $\deg(m_1) \ge 3$. 
If  the $G$-invariant $x_1y_1z_1$ divides $m$ then we are done. Using relation \eqref{eq:tekeres} we may assume that 
$\Phi(m_1)$ contains at least two different weights, say $x_1y_1^2 \mid m_1$. 
Suppose that the multiplicity of $b$ is at least $3$ in $\Phi(m)$; 
then the remainder $m/x_1y_1^2y_i$ must contain an $A$-invariant divisor $w$ with $\deg(w)=2$. Set $v:=y_1y_i$ and $u:=m/vw$ so that $u$ is divisible by $x_1y_1$. 
By Lemma \ref{lemma:trukk} we can replace $m$ with the monomial $uv^gw^{g^2}$, which is divisible by the $G$-invariant $x_1y_1z_1$. 
Finally, if the multiplicity of $b$ in $\Phi(m)$ is $2$, then the multiplicity of $a$ and $c$ must be even, too. 
Then $\deg(m)\ge 8$ and $m$ has an $A$-invariant factorization $m=uvw$ with $x_1y_1^2\mid u$, 
and $\deg(v)=\deg(w)=2$. By Lemma~\ref{lemma:trukk} $m$ can be replaced by $uv^gw^{g^2}$ or 
$uv^{g^2}w^g$ so that we get back to the case treated before. 

It remains that $\deg(m_+) \le 6$. 
If $\deg(m_0) \geq 3$ then $m_0 \in (R_+)_{\le 3}I$ and we are done. 
So for the rest $\deg(m_0) \le 2$. 
Given that $\davenport(A) = 3$ and $\davenport_2(A) = 5$ by Proposition~\ref{prop:halter-koch},
we have $m \in I_1 (I_+)_{\le 3}^3I$ or $m \in I_1^2 (I_+)_{\le 3}^2I$. 
In both cases $m \in I_+^4$ hence $m \in I_+R_+$ by Proposition~\ref{prop:knopeta}. 
Taking into account that $\deg(m) \le 8$ we conclude that $m \in (R_+)_{\le 4}I+ (I_+)_{\le 4}R$, as claimed. 
\end{proof}

\begin{theorem} \label{thm:betaka4} 
If $\mathrm{char}(\field) \neq 2,3$ then 
$\beta_k(A_4) = 4k+2$. 
\end{theorem} 

\begin{proof}
We pove first that $\beta(A_4) \le 6$. To this end consider the subalgebra $S:=\field[U\oplus V^{\oplus 3}]^G$ 
in  $R = \field[U \oplus V^{\oplus n}]^G$ where $n \ge 3$. 
Note that $\beta(S)\le 6$ by Proposition~\ref{prop:etaa4} and in addition $\beta(G)\le \davenport_3(A)=7$ by Corollary~\ref{cor:G:H+1} 
and Proposition~\ref{prop:halter-koch}. 
We have $R_d=\field [\GL_n\cdot S_d]$ for all $d$ if $\mathrm{char}(\field)=0$ by Weyl's Theorem on polarization (cf. \cite{weyl}) 
and in positive characteristic for $d\le \dim(V)(\mathrm{char}(\field)-1) $  by Theorem 5.1 and formula (6.3) in \cite{knop}; in our case $\dim(V)(\mathrm{char}(\field)-1)\ge 12$.
It follows that $R_7= \field [\GL_n\cdot S_7] \subseteq \GL_n\cdot S_+^2\subseteq R_+^2$, whence $\beta(A_4)\le 6$, indeed.

For the rest it suffices to prove that $R_{\ge 7}\subseteq (R_+)_{\le 4}R$ holds for $n\ge 3$, as well,
because then by induction on $k$ we get $R_{\ge 4k+ 3} \subseteq (R_+)_{\le 4}^k R_+$.
Since $R$ is generated by elements of degree at most $6$, 
it is enough to prove that $\bigoplus_{d=7}^{12}R_d\subseteq (R_+)_{\le 4}R$. 
Applying polarization as above and Proposition~\ref{prop:etaa4} we get
$\bigoplus_{d=7}^{12}R_d\subseteq \field[ \GL_n\cdot \bigoplus_{d=7}^{12}S_d] =\field [\GL_n\cdot (S_+)_{\le 4}S] \subseteq (R_+)_{\le 4}R$. 

To prove $\beta_k(A_4)\ge 4k+2$ take as $V$ the natural $4$-dimensional permutation representation of the symmetric group $S_4$. 
It is well known that $R:=\field[V]^{A_4}$ has the Hironaka decomposition $R=P\oplus sP$, where $P$ is the subalgebra generated by the elementary symmetric polynomials 
$p_1,p_2,p_3,p_4$, and $s$ is the degree $6$ alternating polynomial. It is easy to deduce from the Hironaka decomposition that $sp^{k-1}\notin R_+^{k+1}$. 
\end{proof} 

\begin{remark} 
Working over the field of complex numbers Schmid \cite{schmid} already gave a computer assisted proof of the equality  $\beta(A_4, U \oplus V^{\oplus 2})=6$. 
\end{remark}

\begin{corollary}\label{cor:betaa4tilde} 
Suppose that  $\mathrm{char} (\field) \neq 2, 3$. Then  $\beta(\tilde A_4)=12$. 
\end{corollary}

\begin{proof}
We have $\beta(A_4)= 6$  by Theorem~\ref{thm:betaka4}, and since 
$\tilde A_4$ has a two-element normal subgroup $N$ with $\tilde A_4/N\cong A_4$, the inequality $\beta(\tilde A_4)\leq 12$ follows by Lemma~\ref{lemma:red}. 
 It is sufficient to prove the reverse inequality for the field $\mathbb{C}$ 
(as $\beta(G,\mathbb{C}) \le \beta(G,\field)$ by Theorem 4.7 in \cite{knop}). 
Consider the ring of invariants of the $2$-dimensional complex representation of $\tilde A_4$ realizing it as the binary tetrahedral group. 
It is well known (see  the first row in the table of  Lemma 4.1 in \cite{huffman} or Section 0.13 in \cite{popov-vinberg}) 
that this algebra is minimally generated by three elements of degree $6,8,12$, 
whence $\beta(\tilde A_4)\ge 12$. 
\end{proof}


\subsection{The group $(Z_2\times Z_2)\rtimes Z_9$}\label{sec:z2z2z9} 

\begin{proposition}\label{prop:z2z2z9}
Let $G:=(Z_2\times Z_2)\rtimes Z_9$ be the non-abelian semidirect product, 
and suppose that $\mathrm{char}(\field) \neq 2,3$. Then we have $\beta(G)\leq 17$. 
\end{proposition} 

Let $\hat{K} \cong Z_2 \times Z_2= \{ 0,a,b,c \}$ and $Z_9 = \langle g \rangle$. 
Then conjugation by $g$ permutes $a,b,c$ cyclically, say $a^g=b$, $b^g=c$, $c^g=a$. 
$G$ contains the distinguished abelian normal subgroup $A:= K \times C$ where $C:= \langle g^3\rangle\cong Z_3$.  
The conventions of Section~\ref{sec:goebel} can be applied for $(G,A)$, since every irreducible representation of $G$ is $1$-dimensional or is induced from a $1$-dimensional representation of $A$. 
For an arbitrary $G$-module $W$ we set
$J = \field[W]^{C}$, $I= \field[W]^A$, $R=\field[W]^G$;
we  use the transfer maps $\mu:=\tau^G_C:  J \to R$, $\tau:=\tau^G_A: I \to R$. 
For any sequence $S$ over $\hat{A}$ we denote by $S|_C$ 
the sequence obtained from $S$ by restricting  to $C$ each element $\theta \in S$.

\begin{proof} Since $G/C\cong A_4$ and $\beta(A_4)=6$,  by Lemma~\ref{lemma:red}  we have $\beta(G)\le 18$. 
Therefore by Lemma~\ref{goebel} it is sufficient to show that if $m \in I$ is a terminal monomial of degree $18$, then $\tau(m)\in R_+^2$.
We may restrict our attention to the case when $\Phi(m)|_C=(h^{18})$ for a generator $h$ of $\hat C$, 
as otherwise  $m\in J_+^7$, and we get that $\tau(m)=\frac 14\mu(m)\in R_+^2$ by Proposition~\ref{prop:altnoetnum}  applied for $G/C$ acting on $J$.
 We claim that in this case $\Phi(m)$  contains at least $2$ zero-sum sequences 
 of length at most $3$,
whence $m \in I_+^4$ (since $\beta(A)=7$ by Proposition~\ref{prop:halter-koch}), 
and consequently $\tau(m) \in R_+^2$ again by Proposition~\ref{prop:altnoetnum}. 

To verify this claim, factor $m=uv$ where  $\Phi(v)|_K=(0^n)$ and $\Phi(u)|_K$ does not contain $0$. 
If $n\ge 3s$ then $\Phi(v)$ contains at least $s$ zero-sum sequences of length at most $3$. 
Therefore it suffices to show that  $\Phi(u)|_K$ contains the subsequence $(a,b,c)$ whenever $\deg(u) \ge 13$, 
because then the corresponding subsequence of $\Phi(u)$ is a zero-sum sequence over $A$. 
Suppose indirectly that this is false and that $\Phi(u)|_K$ contains e.g. only $a$ and $b$. 
This means that $\Phi(u)|_K= (a^{2x}, b^{2y})$ where $2(x+y) =\deg(u)$.  
By symmetry we may suppose that $x\ge y$ and consequently $x \ge 4$. 
Now $\Phi(u)|_K$ decomposes as follows:  
\begin{align*} 
(a^4, b^2) &\cdot (a^{2x-4}, b^{2y -2})  & \text{ if } y \ge 2;\\
(a^6)   &\cdot (a^{2x-6}, b^{2y}) & \text{ if } y \le 1.
\end{align*}
Observe that the first factor has degree $6$, hence it corresponds to a zero-sum sequence over $\hat A$, and it is a good divisor in the sense of  
Definition~\ref{def:terminal}. This contradicts the assumption that $m$ was terminal. 
\end{proof}


\section{Classification of the groups with large Noether number}

\subsection{A structure theorem}

The  objective of this section is to prove the following purely group theoretical structure theorem:

\begin{theorem} \label{thm:structure}
For any finite  group $G$  one of the following ten options holds:
\begin{enumerate}
\item $G$ contains a cyclic subgroup of index at most $2$;
\item $G$ contains a subgroup isomorphic to:
\begin{enumerate}
\item $Z_2 \times Z_2 \times Z_2$;
\item $Z_p \times Z_p$, where $p$ is an odd prime;
\item $A_4$ or $\tilde{A}_4$;
\end{enumerate}
\item $G$ has a subquotient isomorphic to:
\begin{enumerate}
\item an extension of $Z_2 \times Z_2$ by $Z_2 \times Z_2$;
\item a non-abelian semidirect product $Z_p \rtimes Z_q$ with odd primes $p,q$;
\item $Z_p \rtimes Z_4$, where $p$ is an odd prime and $Z_4$ acts faithfully on $Z_p$;
\item $D_{2p} \times D_{2q}$, where $p,q$ are distinct odd primes;
\item an extension of $D_{2n}$ by $Z_2 \times Z_2$, where $n$ is odd;
\item the non-abelian semidirect product  $(Z_2 \times Z_2) \rtimes Z_9$.
\end{enumerate}
\end{enumerate}
\end{theorem}

\begin{lemma}[Burnside] \label{burn_2-sylow}
If  the Sylow $2$-subgroup $P$ of a group $G$ is cyclic then $G = N \rtimes P$ where 
$N$ is the characteristic subgroup of $G$ consisting of its odd order elements.
\end{lemma}

\begin{proposition}[Zassenhaus, Satz 6 in \cite{zassenhaus35}]\label{zassenhaus}
Let $G$ be a finite solvable group with a Sylow $2$-subgroup $P$ containing a cyclic subgroup of index $2$.
Then $G$ has a normal subgroup $K$ with a cyclic Sylow $2$-subgroup such that  $G/K$ is isomorphic to
one of the groups $Z_2$, $A_4$ or $S_4$. 
\end{proposition}

\begin{lemma}[Roquette \cite{roquette}, or \cite{berk} Lemma 1.4 or \cite{huppert} III. 7. 6 ] \label{noZpZp}
If $G$ is a finite $p$-group which does not contain $Z_p \times Z_p$ as a normal subgroup, 
then either $G$ is cyclic or $p=2$ and $G$ is isomorphic to one of the groups $D_{2^n}, SD_{2^n}, Dic_{2^n}$, where $n>3$,
or to the quaternion group $Q= Dic_{2^3}$. 
\end{lemma}

\begin{corollary} \label{2class}
Any finite $2$-group $G$  falls under case (1), (2a) or (3a) of Theorem~\ref{thm:structure}.
\end{corollary}

\begin{proof}
Suppose that (1) does not hold for $G$. 
Then by Lemma~\ref{noZpZp}, $G$ has a normal subgroup $N \cong Z_2 \times Z_2$. 
Consider the factor group $G/N$: if it is cyclic, i.e. generated by $aN$ for some $a \in G$,
then necessarily $\langle a \rangle \cap N = \{1 \}$, for otherwise $\langle a \rangle$ would be a cyclic subgroup of index $2$ in $G$. 
Now we can find a subgroup $Z_2 \times Z_2 \times Z_2$, which is case (2a):
if $a^2 \neq 1$ then this is because $a^2$ necessarily centralizes $N$,
and if $a^2 =1$ then already $a$ must centralize $N$, for otherwise $G = (Z_2 \times Z_2) \rtimes Z_2 \cong D_8$,
which has a cyclic subgroup of index $2$, a contradiction.  

It remains that $G/N$ is non-cyclic.
 If $G/N$ contains a 
 subgroup isomorphic to $Z_2\times Z_2$, then we get case (3a). 
Otherwise by Lemma~\ref{noZpZp} $G/N$ contains a cyclic subgroup of index $2$.
Given that the Frattini subgroup $F/N$ of $G/N$ is cyclic, $F$ is an extension of a  cyclic group by $Z_2\times Z_2$,
hence by the same argument as above,  $F$ (and hence $G$) falls under case (2a), 
unless $F$ is a non-cylic group with a cyclic subgroup of index $2$. Then $G/\Phi$ (where $\Phi$ is the Frattini subgroup of $F$)
is an extension of $F/\Phi\cong Z_2\times Z_2$ by $G/F\cong Z_2\times Z_2$, and we get case (3a).
\end{proof}

\begin{proposition}\label{allcyclic}
Let $G$ be a group of odd order all of whose Sylow subgroups are cyclic. 
Then either $G$ is cyclic or it falls under case (3b) of Theorem~\ref{thm:structure}. 
\end{proposition}

\begin{proof}
By a theorem of Burnside (see p. 163 in \cite{burnside}) $G$ is isomorphic to $Z_n \rtimes Z_m$ for some coprime integers $n,m$. 
Hence either $G$ is cyclic, or this semidirect product is non-abelian. In the latter case there are elements 
$a \in Z_n$ and $b \in Z_m$ of prime-power orders $p^k$ and $q^r$, which do not commute. 
After factorizing by the centralizer of $\langle a \rangle$ in $\langle b \rangle$ we may suppose that $\langle b \rangle$ acts faithfully on $\langle a \rangle$. 
Then the order $p$ subgroup of $\langle a\rangle$ and the order $q$ subgroup of $\langle b\rangle$ generate a non-abelian semidirect product $Z_p\rtimes Z_q$.
\end{proof}

\begin{proposition}\label{prop:ZnP}
Let $G= Z_n \rtimes P$, where $n$ is odd and $P$ is a $2$-group with a cyclic subgroup of index $2$. 
Then $G$ falls under case (1), (3c), (3d), or (3e) of Theorem~\ref{thm:structure}. 
\end{proposition}

\begin{proof}
Let $C$ be the centralizer of $Z_n$ in $P$. The factor
$P/C$ is isomorphic to a subgroup of $\Aut(Z_n)$, which is abelian, and $G/C = Z_n \rtimes (P/C)$. 
If $P/C$ contains an element of order $4$, then by a similar argument as in Proposition~\ref{allcyclic} we find 
a subquotient isomorphic to $Z_p \rtimes Z_4$, where $Z_4$ acts faithfully on $Z_p$, which is case (3c). 
Otherwise $P/C$ must be isomorphic to $Z_2$ or $Z_2 \times Z_2$. 
If $P/C = Z_2$ then either $C$ is cyclic, and $Z_n \times C$ is a cyclic subgroup of index $2$ in $G$ | this is case (1); 
or else $C$ is non-cyclic, and then $G/\Phi(C)$ (where $\Phi(C)$ is the Frattini subgroup of $C$) is an extension of the dihedral group $G/C \cong D_{2n}$ by the Klein four-group $C/\Phi(C) \cong Z_2 \times Z_2$
| this is case (3e). 

Finally, if $P/C\cong Z_2\times Z_2$,  we  get case (3d): indeed,
$Z_n=P_1\times\dots\times P_r$, where the $P_i$ are the Sylow subgroups of $Z_n$.
If the generators $a$ and $b$ of $Z_2 \times Z_2$ are acting non-trivially on precisely the same set of subgroups $P_i$, 
then since the only involutive automorphism of an odd cyclic group is inversion, $ab$ will act trivially on all $P_i$, hence $ab \in C$, a contradiction.
Therefore a $P_i$ exists such that $a$ acts non-trivially, while $b$ acts trivially on it. 
But an index $j \neq i$ also must exist such that $b$ is acting non-trivially on $P_j$;
after eventually exchanging $a$ with $ab$ we may  suppose that $a$ acts trivially on $P_j$.
Then $G$ has a subfactor $(P_i \times P_j) \rtimes (Z_2 \times Z_2) \cong D_{2p^k} \times D_{2q^l}$, which leads to case (3d).
\end{proof}

\begin{proof}[Proof of Theorem~\ref{thm:structure} for solvable groups]
We shall argue by contradiction: let $G$ be a counterexample of minimal  order. 
Since $G$ does not fall under case (2b), all its odd order Sylow subgroups are cyclic by Lemma~\ref{noZpZp}.
As $G$ does not fall under case (1) or (3b), its order is even by Proposition~\ref{allcyclic}.
Finally, as $G$ does not fall under case (2a) or (3a), its Sylow $2$-subgroup contains a cyclic subgroup of index $2$ by Corollary~\ref{2class}. 
Therefore  Proposition~\ref{zassenhaus} applies to $G$, so a normal subgroup $K$ exists such that $G/K$ is isomorphic to $Z_2$, $A_4$ or $S_4$,
and using Lemma~\ref{burn_2-sylow}, $K=N\rtimes Q$, where $Q$ is a cyclic 2-group while
$N$ is a characteristic subgroup consisting of odd order elements, 
which is also  cyclic, for otherwise it would fall under case (3b). 
The case $G/K\cong S_4$ is ruled out by the minimality of $G$ 
(since otherwise the subgroup $H$ of $G$ with $H/K\cong A_4$ would fall under case (1), a contradiction). 
The case $G/K \cong Z_2$ is also ruled out, since then $G \cong Z_n \rtimes P$
where the Sylow $2$-subgroup $P$ of $G$ has a cyclic subgroup of index $2$,
so it falls under case (1), (3c), (3d), or (3e) by Proposition~\ref{prop:ZnP}.

It remains that $G/K\cong A_4$. 
Suppose first that $N$ is trivial. Then $K=Q$ and $P/Q\cong Z_2 \times Z_2$ is normal in $G/Q\cong A_4$, 
hence $P$ is normal in $G$ and by the Schur-Zassenhaus theorem $G = P \rtimes Z_3$. 
Let $\langle a \rangle$ be the cyclic subgroup of index $2$ in $P$:
the subgroup $\langle a^4 \rangle$ has no non-trivial odd order automorphism, 
hence the factor group $P/\langle a^4 \rangle$ must have a non-trivial automorphism of order $3$. 
But unless $P$ coincides with the group $Z_2 \times Z_2$ or $Dic_8$, 
the factor $P/\langle a^4\rangle$ is isomorphic to $D_8$ or $Z_4 \times Z_2$, which do not have an automorphism of order $3$ 
(for a list of the $2$-groups with a cyclic subgroup of index $2$ see \cite{brown}).  
It follows that $G = (Z_2 \times Z_2) \rtimes Z_3 = A_4$ or $G=Dic_8 \rtimes Z_3 \cong \tilde{A}_4$, 
which is case (2c), a contradiction. 

Finally, suppose that $N$ is nontrivial. 
Since $N$ is characteristic in $K$, it is normal in $G$,
and $G/N$ is isomorphic to $A_4$ or $\tilde{A}_4$ by our previous argument. 
Then $N$ is necessarily cyclic of prime order, for otherwise a proper subgroup $M \le N$ would exist which is normal is $G$, 
and $G/M$ would contain a cyclic subgroup of index at most $2$ by the minimality assumption on $G$, 
but this is impossible since $A_4$ is a homomorphic image of $G/M$. 
Consequently it also follows that $N=Z_3$, for otherwise $|N|$ and $|G/N|$ are coprime, 
so that $G = N \rtimes (G/N)$  by the Schur-Zassenhaus theorem,
and again $G$ would fall under case (2c), a contradiction. 
Let $C$ denote the centralizer of $N$ in $G/N$: on one hand $G/C$ must be isomorphic to a subgroup of $\Aut(Z_3) = Z_2$,
but on the other hand $Z_2$ is not a homomorphic image of $A_4$ or $\tilde{A}_4$, hence $G=C$. 
This means that $N$ is central in $G$, and therefore the Sylow $2$-subgroup $P$ is normal in $G$.
Given that the Sylow $3$-subgroup of $G$ is cyclic and of order $9$ we conclude that $G = P \rtimes Z_9$
where $P$ equals $Dic_8$ or $Z_2 \times Z_2$, and this gives case (3f), a contradiction. 
\end{proof}

\begin{proof} [Proof of Theorem~\ref{thm:structure} for non-solvable groups]
Suppose to the contrary that Theorem~\ref{thm:structure} fails for a non-solvable group $G$, 
which has minimal order among the groups with this property. 
Then any proper subgroup $H$ of $G$ is solvable: indeed, otherwise (2) or (3) of  Theorem~\ref{thm:structure} holds for $H$,  hence also for $G$, a contradiction. 
It follows that $G$ has a solvable normal subgroup $N$ such that $G/N$ is a minimal simple group (i.e. all proper subgroups of $G/N$ are solvable). 
 If $G/N\cong A_5$, then denote by $H$ the inverse image in $G$ of the subgroup $A_4\subseteq A_5$ under the natural surjection $G\to G/N$. 
Then $H$ is solvable, and has $A_4$ as a factor group. Thus $H$ has no cyclic subgroup of index at most two. 
Therefore by the solvable case of Theorem~\ref{thm:structure},  
(2) or (3) holds for $H$, hence it holds also  for $G$, a contradiction. 

According to Corollary 1 in \cite{thompson},  
any minimal simple group is isomorphic to one of the following:
\begin{enumerate} 
\item[(a)] $L_2(2^p)$, $p$ any prime. 

\item[(b)] $L_2(3^p)$, $p$ any odd prime. 

\item[(c)] $L_2(p)$,  $p>3$  prime with $p^2+1\equiv 0$  (mod $5$). 

\item[(d)] $Sz(2^p)$, $p$ any odd prime. 

\item[(e)] $L_3(3)$.
\end{enumerate} 

The group $L_2(2^2)$ is isomorphic to the alternating group $A_5$. 
Finally we show that for the remaining minimal simple groups  (2a), (2b), or (3) holds, hence $G/N$ can not be isomorphic to any of them 
(note that if   (2a), (2b), or (3) holds for $G/N$, then (2a), (2b), or (3) holds for $G$ by Sylow's theorem, Lemma~\ref{noZpZp} and Corollary~\ref{2class}).  

The group $L_2(2^p)$ contains as a subgroup 
the additive group of the field of $2^p$ elements.  Hence when $p\ge3$ then (2a) holds. 
Similarly, $L_2(3^p)$ contains as a subgroup the additive group of the field of $3^p$ elements, hence (2b) holds.  
The subgroup of unipotent upper triangular matrices in $L_3(3)$ is a non-abelian group of order $27$, hence (2b) holds for it. 
The subgroup in $SL_2(p)$ consisting of the upper triangular matrices is isomorphic to the semidirect product $Z_p\rtimes Z_{p-1}$. Its image in $L_2(p)$ contains the non-abelian semidirect product $Z_p\rtimes Z_q$ for any odd prime divisor $q$  of $p-1$. When $p$ is a Fermat prime, then $L_2(p)$ contains $Z_p\rtimes Z_4$ (where $Z_4$ acts faithfully on $Z_p$), except for $p=5$, but we need to consider only primes $p$ with  $p^2+1\equiv 0$  (mod $5$). 
The Sylow $2$-subgroup of $Sz(q)$ is a so-called Suzuki $2$-group of order $q^2$, that is,  a non-abelian 2-group with more than one involution, having a cyclic group of automorphisms which permutes its involutions transitively. Its center consist of the involutions plus the identity, and it has order $q$, see for example \cite{higman}, \cite{collins}. 
It follows that  the Sylow $2$-subgroup $Q$ of $Sz(2^p)$ ($p$ an odd prime) properly 
contains an elementary abelian $2$-group of rank $p$ in its Sylow $2$-subgroup, hence (2a) holds for it.  
\end{proof}


\subsection{Proof of the classification theorem} \label{sec:proof} 

\begin{proof}[Proof of Theorem \ref{thm:main}] 
It suffices to consider the cases listed in Theorem~\ref{thm:structure}:

 \begin{enumerate}  
 \item if $G$ contains a subgroup of index at most $2$ then $\beta(G) \ge \frac{1}{2}|G|$ by Proposition 5.1 in \cite{schmid} 
 (in fact $\beta(G)-\frac 12|G|\in\{1,2\}$ by \cite{CzD:2}).
  
 \item if $G$ contains a subgroup $H$ of index $k$ such that:
 
 \begin{enumerate}
 \item $H \cong Z_2 \times Z_2 \times Z_2$ then by Proposition~\ref{prop:Z2Z2Z2} and Corollary~\ref{cor:G:H+1}
 \[ \frac{\beta(G)}{|G|} \le \frac{1}{8k} \beta_k(Z_2 \times Z_2 \times Z_2)  = \frac{1}{4} + \frac{3}{8k}. \] 
 \item $H \cong Z_p \times Z_p$ then by Proposition~\ref{prop:halter-koch} and Corollary~\ref{cor:G:H+1}
 \[ \frac{\beta(G)}{|G|} \le \frac{1}{kp^2} \beta_k(Z_p \times Z_p) = \frac{1}{p} + \frac{p-1}{kp^2}. \] 
 \item $H \cong A_4$  then by Theorem~\ref{thm:betaka4}  and Corollary~\ref{cor:G:H+1}
 \[ \frac{\beta(G)}{|G|} \le \frac{1}{12 k} \beta_k(A_4) = \frac{1}{3} + \frac{1}{6k}. \]
 \end{enumerate}
It is easily checked that in all three cases the inequality $\frac{\beta(G)}{|G|} \ge \frac{1}{2}$  holds if and only if $k=1$, 
and in case (b) it is also necessary that $p=2 \text{ or } 3$.
 Finally, let $H = \tilde{A}_4$;  by Lemma~\ref{lemma:myred} we have $\beta_k(\tilde{A}_4) \le 2 \beta_k(A_4)$
 hence  $\beta(G) \le \beta_{k}(\tilde{A}_4) \le 8k+4$ by Corollary~\ref{cor:G:H+1} and Theorem~\ref{thm:betaka4},
so we get the same upper bound on $\beta(G)/|G|$ as in the case when $H = A_4$. 
 
   \item
  For any subquotient $K$ of $G$ we have   $\beta(G)/|G| \le \beta(K)/|K|$ by Lemma~\ref{lemma:red}; 
 \begin{enumerate}
 \item if $K/N \cong Z_2\times Z_2$ for some normal subgroup $N \cong Z_2\times Z_2$  
 then by Lemma~\ref{lemma:myred} and Proposition~\ref{prop:halter-koch}:
 \[ \frac{\beta(K)}{|K|} \le \frac{1}{16}\beta_{\beta(Z_2 \times Z_2)} (Z_2 \times Z_2) =
\frac{1}{16}\beta_3 (Z_2 \times Z_2)  = \frac{7}{16}.\]
  \item if $K \cong Z_p\rtimes Z_q$ then $\beta(K)/|K| < \frac{1}{2}$ by Theorem~\ref{thm:zpzq}.  
\item if $K \cong Z_p\rtimes Z_4$, where $Z_4$ acts faithfully, then  by Corollary~\ref{cor:skatulya} 
\[\frac{\beta(K)}{|K|}\le  \frac{p+6}{4p} \le \frac{13}{28}\]
for $p \ge 7$, and $\beta(K)/|K| = 2/5$ for $p=5$ by Proposition~\ref{prop:z5z4}.
\item if $K \cong D_{2p}\times D_{2q}$  where $p,q$ are distinct odd primes (hence $p \ge 3$ and $q \ge 5$) then by Lemma~\ref{lemma:red} and Proposition~\ref{prop:betaD2n}:
\[ \frac{\beta(G)}{|G|} \le \frac{1}{4pq} \beta(D_{2p}) \beta(D_{2q}) \le \frac{ (p+1)(q+1) }{4pq}  \le \frac 2 5.\]
 \item if $K/N \cong D_{2p}$  for some normal subgroup $N \cong Z_2\times Z_2$
  then by Lemma~\ref{lemma:myred} and Proposition~\ref{prop:betaD2n}: 
 \[  \frac{\beta(G)}{|G|}   \le \frac{1}{8p} \beta_{\beta(D_{2p})} (Z_2 \times Z_2) \le \frac{2p+3}{8p} \le \frac{3}{8}.\] 
 \item if $K \cong (Z_2\times Z_2)\rtimes Z_{9}$ then $\beta(K)/|K| \le \frac{17}{36}$ by Proposition~\ref{prop:z2z2z9}. 
 \end{enumerate}
To sum up, $\beta(G)/|G| < 1/2$ whenever $G$ falls under case $(3)$ of Theorem~\ref{thm:structure}. \qedhere
\end{enumerate}
\end{proof}


\bibliographystyle{cdraifplain}

\end{document}